\documentclass[11pt]{amsart}
\usepackage{graphicx} 
\usepackage{xcolor}
\usepackage{a4wide}
\usepackage{amssymb}

\definecolor{cipturq}{RGB}{0, 180, 158}

\definecolor{ser}{RGB}{200, 0, 0}

\usepackage{comment}

\usepackage{times}
\newtheorem{proposition}{Proposition}
\newtheorem{lemma}[proposition]{Lemma}

\newtheorem{corollary}[proposition]{Corollary}
\newtheorem{theorem}[proposition]{Theorem}
\newtheorem{remark}[proposition]{Remark}

\theoremstyle{definition}
\newtheorem{definition}[proposition]{Definition}
\newtheorem{notation}[proposition]{Notation}
\newtheorem*{remark*}{Remark}
\newtheorem*{example*}{Example}

\newcommand{\WL}{\operatorname{WL}}
\newcommand{\Dl}{\left( \frac{Z_t'}{Z_t} \right) (s)}
\newcommand{\Dlsz}{\left( \frac{Z_t'}{Z_t} \right) (s_0)}
\newcommand{\Dd}{\operatorname{D}}
\newcommand{\Rr}{\operatorname{R}} 
\newcommand{\cp}{\operatorname{cp}}  
\newcommand{\lp}{\left(}
\newcommand{\rp}{\right)}
\newcommand{\Diag}{\operatorname{Diag}}
\newcommand{\Area}{\operatorname{Area}}
\newcommand{\Res}{\operatorname{Res}}

\newcommand{\Spec}{\operatorname{Spec}}

\newcommand{\Tr}{\operatorname{Tr}}

\newcommand{\PSL}{\operatorname{PSL}}
\newcommand{\Deck}{\operatorname{Deck}}
\newcommand{\dist}{\operatorname{dist}}
\newcommand{\Diam}{\operatorname{Diam}}
\newcommand{\pr}{\operatorname{pr}}
\newcommand{\sys}{\operatorname{sys}}

\newcommand{\ad}{\operatorname{Ad}}
\newcommand{\Range}{\operatorname{Range}}
\newcommand{\Id}{\operatorname{Id}}
\newcommand{\Ll}{\operatorname{L}}
\newcommand{\Hh}{\operatorname{H}}
\newcommand{\End}{\operatorname{End}}
\newcommand{\cgot}{\mathfrak{C}}

\newcommand{\hz}{\mathbb{H}}
\newcommand{\sz}{\mathbb{S}}
\newcommand{\rz}{\mathbb{R}}
\newcommand{\cz}{\mathbb{C}}
\newcommand{\nz}{\mathbb{N}}
\newcommand{\cM}{\mathcal{M}}
\newcommand{\cC}{\mathcal{C}}
\newcommand{\cO}{\mathcal{O}}
\newcommand{\cT}{\mathcal{T}}
\newcommand{\zz}{\mathbb{Z}}
\newcommand{\Spin}{\mathrm{Spin}}
\newcommand{\spin}{\mathrm{spin}}
\newcommand{\SO}{\mathrm{SO}}
\newcommand{\so}{\mathrm{so}}
\newcommand{\dd}{\mathrm{d}}
\newcommand{\hol}{\mathrm{hol}}
\newcommand{\tGamma}{\tilde{\Gamma}}
\newcommand{\tX}{\tilde{X}}
\newcommand{\tp}{\tilde{p}}
\newcommand{\G}{\mathrm{PSL}_2(\rz)}
\newcommand{\SL}{\mathrm{SL}}
\newcommand{\Hol}{\mathcal{O}}
\newcommand{\Pspd}{\operatorname{P_{Spin(2)}}}
\newcommand{\PSO}{\operatorname{P_{\SO(2)}}}
\newcommand{\ball}{B_{\delta}(\lambda_0)}

\newcommand{\cR}{\mathcal{R}}
\newcommand{\MCG}{\operatorname{MCG}}
\newcommand{\Diff}{\operatorname{Diff}}

\newcommand{\tb}{\operatorname{tb}}

\newcommand{\Xs}{\operatorname{X_{cp}^1}}
\newcommand{\Xd}{\operatorname{X_{cp}^2}}
\newcommand{\ffb}{\operatorname{ff_b}}

\newcommand{\ffc}{\operatorname{ff_c}}
\newcommand{\II}{\mathcal{I}}
\newcommand{\Ss}{\mathcal{S}}
\newcommand{\PP}{\operatorname{P}}

\renewcommand{\spin}{\mathrm{Spin}}
\newcommand{\cpTr}{{}^{\cp}\!\Tr}

\begin{document}
\title[The Selberg Zeta function on the spin moduli space]{Asymptotics of the Selberg Zeta function on the spin moduli space}
\author{Cipriana Anghel}
\address{Mathematisches Institut, Georg-August-Universit\"at G\"ottingen,
Bunsenstr.\ 3-5, 37073 G\"ottingen, Germany and Institutul de Matematic\u a al Academiei Rom\^ane,
calea Grivi\c tei 21, Bucharest, Romania}

\email{cipriana.anghel-stan@mathematik.uni-goettingen.de}

\author{Sergiu Moroianu}
\address{Facultatea de Matematic\u a, strada Academiei 14, Bucure\c sti, and Institutul de Matematic\u a al Academiei Rom\^ane, calea Grivi\c tei 21, Bucharest, Romania}
\email{moroianu@alum.mit.edu}

\author{Rareș Stan}
\address{Institutul de Matematic\u a al Academiei Rom\^ane, calea Grivi\c tei 21, Bucharest, Romania}
\email{rares.stan@imar.ro}
\date{\today}

\begin{abstract}
We prove a limited asymptotic expansion up to order $t^4\log t$ of the logarithmic derivative of the Selberg zeta function for the spin Dirac operator on compact surfaces for families of hyperbolic metrics degenerating towards complete hyperbolic metrics with cusps in a pinching process where the lengths $l_j(t)$ of certain disjoint simple closed geodesics converge smoothly to $0$ as $t\to 0$.
\end{abstract}
\maketitle

\section{Introduction}\label{intro}
\subsection*{Zeta functions}
For every sequences of complex numbers $(l_j)$, $(\varepsilon_j)$, $j\geq 1$, consider the infinite product depending on a complex parameter $s$:
\begin{equation}\label{genzeta}
Z \lp s,(l, \varepsilon) \rp:=\prod_{j\geq 1} \prod_{m=0}^{\infty} \lp 1- \varepsilon_j e^{-l_j (s+m)}  \rp.
\end{equation}
Under rather mild growth conditions, this infinite product is absolutely convergent in some half-plane $\{\Re(s)\gg 0\}$, where it defines a holomorphic function, the \emph{Selberg zeta function} associated to the sequences $l$ and $\varepsilon$. For instance, when $l_j$ is the logarithm of the $j^{\text{th}}$ prime  and $\varepsilon_j\equiv 1$, the product is absolutely convergent for $\Re(s)>1$, yielding $Z(s)=\left(\prod_{m=0}^{\infty}\zeta(s+m)\right)^{-1}$ where $\zeta$ is Riemann's zeta. In this example, $Z$ extends meromorphically to $\cz$. However, for general sequences, the Selberg zeta function has no apparent reason to extend analytically outside its domain of absolute convergence.

\subsection*{Selberg Zeta functions for hyperbolic surfaces}
In some special cases where the sequences $l$ and $\varepsilon$ arise from geometric data, the function $Z$ does extend analytically to the complex plane. In this paper, $(X,g^X)$ will be an oriented hyperbolic surface, either compact or complete of finite volume, and $l_j=l_j(g^X)$ the $j^{\text{th}}$ length of primitive oriented closed geodesics on $X$ with respect to the metric $g^X$, counted with multiplicity. The developing map of $g^X$ induces a faithful, discrete representation of $\Gamma:=\pi_1(X,p)$ in $\G$, well defined up to conjugacy in $\G$. The set of all oriented closed geodesics on $X$ is in bijection with the set $\{\rho(\Gamma)\}_{\text{hyp}}$ of hyperbolic conjugacy classes in $\rho(\Gamma)$. For any \emph{bounded} class function $\varepsilon:\{\rho(\Gamma)\}_{\text{hyp}}\to \cz$, the Selberg zeta product is absolutely convergent when $\Re(s)>1$. Note that one could relax the boundedness condition when studying zeta functions arising from non-unitary representations, for us however $\varepsilon$ will always take values in $\{\pm 1\}$.

\subsection*{Functionals on the Riemann moduli space} 
For a compact oriented hyperbolic surface $(X,g^X)$ of genus $g$, the classical Selberg zeta function
 is defined by taking $\varepsilon\equiv 1$: 
\begin{equation}\label{szcla}
Z \lp s,(X,g^X) \rp = 
\prod_{j\geq 1} \prod_{m=0}^{\infty} \lp 1- e^{-l_j(g^X) (s+m)}  \rp.
\end{equation}
It extends to an entire function on $\cz$, satisfies a functional equation with respect to the involution $s\mapsto 1-s$, and has non-trivial zeros precisely on the line $\{\Re(s)=1/2\}$. For any diffeomorphism $\Phi$ of $X$, the zeta functions for the metrics $g^X$ and $\Phi^*g^X$ clearly coincide. We get in this way a function defined on the moduli space of Riemann surfaces, taking values in entire functions:
\[Z:\cM_g\to \cO(\cz).
\]
The value at $s=1/2$ is particularly interesting: it is essentially equal to the regularized determinant of the Laplacian, so it provides a K\"ahler potential for the Weil-Petersson metric on $\cM_g$. It is natural to ask whether $Z$ can be extended in any meaningful way to the Deligne-Mumford compactification of the moduli space. This question was answered in the negative by Schulze \cite{Schulze}. The underlying reason seems to be the following: by the Selberg trace formula, the zeta function is related to the trace of a difference of resolvents of the Laplacian on $(X,g^X)$. However, on hyperbolic surfaces with cusps, corresponding to points in the compactification of $\cM_g$ by stable nodal curves, such a relative resolvent is never trace class. 

\subsection*{Selberg zeta functions on the spin moduli space} 
The set of spin structures on a Riemann surface $X$ is in bijection with $H^1(X,\zz/2\zz)$. Every spin structure determines a $\{ \pm 1 \}$-valued function $\varepsilon$ on the set of conjugacy classes in the fundamental group; for instance, for Riemann surfaces whose universal cover is the unit disk, the value of $\varepsilon$ on a hyperbolic conjugacy class is the holonomy in the spin bundle along the corresponding closed geodesic with respect to the unique complete hyperbolic conformal metric on $X$. The \emph{spin Selberg zeta function} is defined for $\Re(s) >1$ by the absolutely convergent product
\begin{equation} \label{szf}
Z \lp s,(X,g, \varepsilon) \rp:=\prod_{\mu\in\{\Gamma\}_{\mathrm{hyp}}} \prod_{m=0}^{\infty} \lp 1- \varepsilon(\mu) e^{-l(\mu) (s+m)}  \rp,  
\end{equation}
where $\mu$ runs along all the oriented primitive closed geodesics on $X$, and $l(\mu)$ is the length of the geodesic $\mu$. It differs from the standard Selberg zeta function \eqref{szcla} by the appearance of the sign factors $\varepsilon(\mu)$. 

For $X$ compact, it follows from the Selberg trace formula that the function $Z$ extends analytically to an entire function, and one can compute determinants of Dirac Laplacians in terms of special values of this function at half-integers \cite{dhocker}, \cite{sarnak}. Since isotopies preserve spin structures, one gets a function on the Teichm\"uller space with values in entire functions. This function descends to the quotient $\cM_g^\spin$ of $\cT_g$ by the finite index subgroup of those modular transformations preserving all spin structures on $X$, a Galois finite cover of $\cM_g$ that we call here the \emph{spin moduli space}:
\[Z:\cM_g^\spin\to \cO(\cz).
\]
It was proved by one of us \cite{rares} that the spin Selberg trace formula, and as a consequence the analytic extension of $Z$, still hold true for complete hyperbolic surfaces of finite area, under a topological assumption on the spin structure along the cusps. The key fact, first noticed by B\"ar \cite{Bar}, is that if $\varepsilon(\mu)=-1$ for every primitive \emph{parabolic} conjugacy class, then the spectrum of the Dirac operator on the surface with cusps is discrete. For a path of hyperbolic metrics on the compact surface $X$ degenerating towards a surface with cusps in a so-called \emph{pinching process}, detailed in Definition \ref{pinchingproc} below, 
Stan \cite[Theorem 5]{rares} proved the uniform convergence on compacts of $\cz$ of the zeta function on $X$, rescaled by a factor depending solely on the length of the systole, towards the zeta function of the limiting surface with cusps. This result will be interpreted below as finding the leading term of the zeta function's asymptotics along certain paths towards the boundary of the compactification of the spin moduli space $\cM_g^\spin$.

\subsection*{The pinching process}\label{pinchingproc}

Let $X$ be an oriented smooth compact surface of genus $g$, and fix a maximal set of mutually disjoint simple closed curves $(\gamma_j)_{1\leq j\leq3g-3}$. By cutting along these curves we obtain a topological pants decomposition of the surface $X$.
Fix $1\leq k\leq 3g-3$, denote by $\gamma$ the union of the curves $\gamma_1,\ldots,\gamma_k$, and set $X_0=X\setminus \gamma$ the complement of $\gamma$ in $X$.

\begin{definition} 
A smooth family of hyperbolic metrics $(g_t)_{t\in (0,1)}$ on $X$ describes a \emph{pinching process} of $X$ along the curve $\gamma:=\bigcup_{j=1}^k\gamma_j$ if the following conditions are satisfied:
\begin{itemize}
\item The possibly disconnected curve $\gamma$ is geodesic with respect to $g_t$, for all $t>0$.
\item There exists a  complete hyperbolic metric $g_0$ on $X_0=X \setminus \gamma$ such that the restriction of $g_t$ to $X_0$ extends in $t=0$ to a smooth family of metrics $\{g_t\}_{t \in [0,1)}$.
\item  For all $j\leq k$, the length function of $\gamma_j$ with respect to the metrics $g_t$,
\begin{align*}l(\gamma_j):(0,1)\to \rz,&&  t\mapsto l_t(\gamma_j)
\end{align*}
 extends smoothly in $t=0$, and moreover $\lim_{t\to 0}l_t(\gamma_j)=0$.
\item Near a geodesic $\gamma_j$, the metrics $g_t$ takes the form
\[ g_t = \frac{\dd x^2}{x^2 + l_t^2(\gamma_j)} + (x^2 + l_t^2(\gamma_j)) \dd\theta^2; \quad (x,\theta) \in \left( -\frac{l_t(\gamma_j)}{\sinh\frac{l_t(\gamma_j)}{2}}, \frac{l_t(\gamma_j)}{\sinh\frac{l_t(\gamma_j)}{2}} \right) \times [0, 2\pi]. \]
\end{itemize}
\end{definition}
In the limit $t=0$, we obtain a possibly disconnected, non-compact surface $X_0$ with $2k$ cusps, endowed with a complete hyperbolic metric $g_0$. By Gauss-Bonnet, the area remains constant throughout this process, including at $t=0$. 
It follows from the definition that $l_t(\gamma_j)$ has a positive limit as $t\to 0$, for all $j\geq k+1$, and that
 $\theta_t(j)$ has a limit when $t \to 0$ for all $j\leq 3g-3$, where $\theta_t(j)$ is the twisting parameter along $\gamma_j$ for the metric $g_t$ in the Fenchel-Nielsen coordinates on the Teichm\"uller space $\cT_g$. 
 
The $C^0$ behaviour of the length spectrum and also of the spectrum of the Laplace and Dirac operators under such degenerations were studied among others by Ji \cite{Ji}, Colbois-Courtois \cite{colbois-courtois}, B\"ar \cite{Bar}, Schulze \cite{Schulze}, Pf\"affle \cite{pfaffle}, and more recently Stan \cite{rares}. It also motivated the construction by Anghel \cite{cipriana} of a pseudodifferential calculus with parameters containing the resolvents of the family of Dirac operators in a pinching process.

\subsection*{Results}
The motivating goal of the present article is to derive a refined asymptotic expansion of the family of Selberg zeta functions along a pinching process near $t=0$. We stress again that the analytic properties needed even for $C^0$ convergence only hold under the topological assumption  that the class function $\varepsilon$ determined by the spin structure does not take the value $+1$ on the free homotopy class of any one among the $k$ pinched geodesics. Several difficulties arise in this task: accounting for the change in the topology of the underlying surface when $t$ -- the degenerating parameter -- becomes $0$; controlling uniformly the derivatives of the lengths of closed geodesics with respect to $t$; and extending smoothness results from the domain of convergence $\{\Re(s)>1\}$ of the Zeta function to the rest of the complex plane. To address these difficulties, we rely heavily on our previous works \cite{rares}, where the leading term in the asymptotic was found, and \cite{cipriana}, which offers the adequate tool for improving the continuity result of \cite{rares}, namely the cusp-surgery pseudodifferential calculus, first developed in \cite{ars} in a slightly more general context. 

As a preliminary step towards our goal, we start by proving carefully that the Selberg zeta function depends not only continuously, but actually \emph{smoothly} on $t$ for $t>0$, and more generally, for any smooth family of hyperbolic metrics indexed by some manifold $\cM$ (Theorem \ref{smdepz}). 

\begin{theorem}\label{thmain1}
Let $X$ be a compact oriented smooth surface of genus $g$. For every fixed spin structure on $X$,
the family of Selberg zeta functions $\{Z_t(s)\}_{t\in\cM_g^\spin}$ defines a smooth function on the spin moduli space with values in entire functions, i.e., it belongs to $\mathcal C^{\infty} \lp \cM_g^\spin, \Hol (\mathbb C) \rp$. 
\end{theorem}

The same result holds, with the same proof, not only for the spin Selberg zeta function \eqref{szf}, but also for the classical zeta function corresponding to the Laplacian, and for their variants when we twist by a flat unitary bundle.

The proof of Theorem \ref{thmain1} is not obvious, in particular it depends crucially on a \emph{uniform} Milnor-\v{S}varc Lemma, allowing us to bound the word length of a conjugacy class $[\eta]$ in the fundamental group of $X$ (see Definition \ref{defwordlength}) in terms of the length of the corresponding geodesic, uniformly for $t$ varying in a compact subset of the parameter space $\cM $. The Selberg trace formula also plays a key role in the proof, suggesting that the smoothness result is in fact unlikely to hold for smooth families of metrics of variable negative curvature, or for twisting functions $\varepsilon$ unrelated to some geometric elliptic differential operators. We would expect it however to hold with only minor changes for zeta functions corresponding to Laplacians twisted by flat bundles with non-unitary holonomy.

Our main result (Theorem \ref{mainth}) is an asymptotic expansion up to terms of order $t^4\log t$ of (the logarithmic derivative of) the zeta function in the pinching process as the systole shrinks to $0$, extending Stan's result. 

Consider the family of holomorphic functions $Z_t(s):=Z(s,(X,g_t,\varepsilon))$ for $t> 0$, together with $Z_0(s)=Z(s,(X_0,g_0,\varepsilon))$ at $t=0$. We have already mentioned that, somewhat unsurprisingly, $Z_t(s)$ is smooth in $(t,s)\in(0,\infty)\times \cz$. Near $t=0$, Stan's formula \cite[Theorem 3.5]{rares} states that 
\[
\lim_{t\to 0}\left(Z_t(s)\exp\left(-\frac{\pi^2}{6}\sum_{j=1}^kl_t(\eta_j)^{-1}\right)
\right) = 2^{k(1-2s)}Z_0(s).
\]
Our main result (Theorem \ref{mainth}) is an asymptotic expansion of (the logarithmic derivative of) the zeta functions in the pinching process as $l_t(\gamma_j)$ shrink to $0$, extending in this sense Stan's result. The price to pay for this extension is one extra derivative in $s$, as we are not able to prove more than $C^0$ convergence of $Z_t'(s_0)/Z_t(s_0)$ for any value of $s_0$.

\begin{theorem}\label{mainth}
Let $R>0$ such that
$\mathcal C_R(0) \cap  \lp  \tfrac{1}{2} + i \Spec \Dd_{0} \rp = \emptyset.$
Then there exists $t_0=t_0(R) >0$ and a rational function $Q_{t,R}(s)$ in $s$ with smooth coefficients in $t\in[0,t_0)$, so that for $(t,s) \in [0,t_0) \times B_R(0)$ we have 
\begin{align*}
{}&\frac{\dd}{\dd s}  \frac{1}{2s-1} \left[ \lp \frac{Z_t'}{Z_t}  \rp(s) -  \lp \frac{Z_0'}{Z_0} \rp (s) +  2k \log 2  - \frac{1}{2} \frac{\dd}{\dd s} \log Q_{t,R}(s) \right]  \\
{}&\in  t \, \mathcal C^{\infty} \lp [0,t_0), \Hol \lp B_R(0) \rp \rp + t^4 \log t \, \mathcal C^{\infty} \lp [0,t_0), \Hol \lp B_R(0) \rp \rp.
 \end{align*}
 \end{theorem}
Here $\mathcal C_R(0)$ and $B_R(0)$ are the circle, respectively the open disc of radius $R$ centered at $0$.
The term $C^{\infty} \lp [0,t_0), \Hol \lp B_R(0) \rp \rp$ denotes a smooth function on the manifold with boundary $[0,t_0)\times B_R(0)$ which is holomorphic in the second variable $s$.

When trying to analyze the zeta function itself, rather than its logarithmic derivative, the key challenge arises from the fact that the Milnor-\v{S}varc lemma no longer holds for surfaces with cusps. Ultimately, we stop short of improving Stan's $C^0$ convergence for $Z_t$ at $t=0$.
We intend to overcome these difficulties in a future work, possibly by using explicit formul\ae \ due to C.\ Series  \cite{series} and S.\ Wolpert \cite{wolpert} for derivatives of lengths of geodesics along certain paths towards the boundary of the Teichm\"uller space.

\subsection*{Acknowledgements}
We are indebted to Ksenia Fedosova for useful discussions. The authors were partially supported from the PNRR project III-C9-2023-I8-CF149. Cipriana Anghel received funding of the Klaus Tschira Boost Fund, a joint initiative of GSO – Guidance, Skills \& Opportunities e.V. and Klaus Tschira Stiftung.

\section{Preliminaries}
\subsection{Spin structures in dimensions $1$ and $2$}\label{epsilon}

The group $\Spin(n)$ is a compact Lie group equipped with a $2:1$ morphism $\Spin(n)\to \SO(n)$. 
For $n\geq 3$, $\Spin(n)$ is the universal cover of $\SO(n)$. In low dimensions this is no longer the case: 
\begin{align*}
\Spin(1)=\{\pm 1\},&& \Spin(2)=\sz^1.
\end{align*}

A \emph{spin structure} on an oriented Riemannian manifold $X$ of dimension $n$ is a $2:1$ covering of the oriented frame bundle $P_\SO X$ by a principal $\Spin(n)$ bundle $P_\Spin X$, compatible with the $2:1$ morphism between the structure groups. The Levi-Civita connection $1$-form lifts from $P_\SO X$ to a $\spin(n)$-valued connection on $P_\Spin X$.
The \emph{spinor bundle} is a complex vector bundle associated to $P_\Spin X$ via the spinor representation of $\Spin(n)$ on $\cz^{2^{[n/2]}}$, obtained by restriction from the standard representation of the Clifford algebra. It has a induced Levi-Civita covariant derivative, and a Clifford action of the (co)tangent bundle of $X$. The \emph{spin Dirac operator}, an elliptic, symmetric, first-order differential operator on spinors, is defined as the composition of the Clifford multiplication with the covariant derivative.

Spin structures always exist on an oriented manifold $X$ of dimension $1$ or $2$, and they are in bijection with the cohomology space $H^1(X,\zz/2\zz)$. 

\subsubsection{Spin structures on the circle}
On the circle $\sz^1=\rz/2\pi \zz$ there exist two spin structures, namely the two covering spaces over $\sz^1$ with Deck transformation group $\{\pm 1\}$. One of those, whose total space consists of two disjoint copies of $\sz^1$, is called the \emph{trivial spin structure on $\sz^1$}. The second one, called \emph{nontrivial}, has total space $\sz^1$, and the covering map is the square map $e^{it}\mapsto e^{2it}$.

The spinor bundle for both spin structures is the trivial line bundle $\cz\times \sz^1\to \sz^1$. The corresponding Levi-Civita connections have holonomy $1$ for the trivial spin structure and $-1$ for the nontrivial one. The associated Dirac operators are
$i\frac{\dd}{\dd t}$, respectively $i\frac{\dd}{\dd t}+\frac{1}{2}$. From an analytic point of view, following B\"ar \cite{Bar}, it is important to note that the first of these two operators has a nontrivial nullspace.

\subsubsection{Spin structures on surfaces}
The $2:1$ morphism $\pi:\Spin(2)\to \SO(2)$ is given by
\[
e^{i\theta}\mapsto 
\begin{bmatrix} 
\cos (2\theta) & \sin(2\theta) \\
-\sin(2\theta)& \cos(2\theta)
\end{bmatrix}.
\]
At the level of Lie algebras, the induced isomorphism $\pi_*$ is
\[
i\rz\ni it\mapsto \begin{bmatrix} 
0&2 \\
-2& 0
\end{bmatrix}\in\so(2).
\]

On a compact oriented surface of genus $g$ there are $2^{2g}$ spin structures. If the surface has $n\geq 1$ punctures, this number becomes $2^{2g+n-1}$. 
For each spin structure $P_\Spin X$, the associated line bundles $P_\Spin X\times_{\rho^\pm}\cz$ with respect to the unitary representations of rank $1$
\[
\rho^\pm(e^{it})= \begin{bmatrix} 
\cos (\theta) & \pm\sin(\theta) \\
\mp\sin(\theta)& \cos(\theta)
\end{bmatrix}
\]
are the so-called positive, respectively negative \emph{chiral spinor bundles} $S^\pm X\to X$. Their direct sum is the spinor bundle on $X$, denoted $SX$.
The Levi-Civita covariant derivative on $TX$ can be interpreted as a $\so(2)$-valued $\ad$-covariant $1$-form on the unit frame bundle $P_\SO X$. Its lift to $P_\Spin X$, composed with the inverse of the isomorphism $\pi_*: i\rz\to \so(2)$, is a connection $1$-form on $P_\Spin X$. The Levi-Civita covariant derivative on spinors is defined as the covariant derivative on $S^\pm X$ associated to this connection and to the spinor representations $\rho^\pm$.

Let $\gamma$ be an immersed closed geodesic in $X$. The tangent vector field to $\gamma$ and its rotation by $\pi/2$ form a parallel orthonormal frame, so the holonomy along $\gamma$ in the frame bundle is trivial. Any local lift of a parallel frame is a parallel section in $P_\SO X$. It follows that the holonomy $\hol^\Spin(\gamma)$ along the geodesic $\gamma$ in the spin bundle is either $1$ or $-1$. 

\begin{definition}\label{defep}
For every immersed closed geodesic $\gamma$ in $X$, set
\[
\varepsilon(\gamma)= \hol^\Spin(\gamma)\in \{\pm 1\}.
\]
\end{definition}
This sign function defined on the set of closed geodesics, and its extension to some "almost-geodesic" curves on manifolds with cusps, plays an important role in the definition of the Selberg Zeta function associated to the spin structure. In particular, we will only deal in this paper with shrinking simple closed geodesics on which the $\varepsilon$ function takes the value $-1$. Equivalently, we shall assume that the spin structure induced on the simple closed geodesics shrinking to $0$ is nontrivial. 

\subsubsection{Spin structures on complete hyperbolic surfaces}
A hyperbolic surface is a Riemannian $2$-fold with Gauss curvature $-1$. By a classical result due to Hopf, the universal cover $\tilde X$ of a complete hyperbolic surface $(X,g)$ is isometric to the hyperbolic upper-half plane $\hz$, and the isometry is unique up to post-composing with an isometry of $\hz$. An oriented isometry $\Phi:\tilde X\to\hz$ induces an isomorphism of $\SO(2)$-principal bundles between the oriented orthonormal frame bundles $P_\SO\tX\to \tX$ and $P_\SO\hz\to\hz$. Recall also that the group of oriented isometries of $\hz$ is 
$\G=\SL_2(\rz)/\{\pm 1\}$, and that this group acts transitively on the set of oriented orthonormal frames on $\hz$. It follows that, after fixing an orthonormal frame $(\partial_x,\partial_y)$ at $i\in\hz$, we obtain an identification of $P_\SO\hz$ with the group $\G$ as $\SO(2)$ principal bundles. Then the canonical projection $\SL_2(\rz)\to \G$ satisfies the requirements to be a spin structure on $\hz$, where $\SL_2(\rz)$ is equipped with the standard right action of the group $\Spin(2)=\sz^1$ by rotations. Since $\hz$ is contractible, its first cohomology group with coefficients in $\zz/2\zz$ vanishes, so this spin structure is unique.

Let $\Gamma=\pi_1(X,p)$ denote the fundamental group of $X$ for a fixed base point $p$. Fix also a preimage $\tp$ of $p$ in the universal cover $\tX$ (for instance, take $\tp$ to be the relative homotopy class of the constant loop at $p$). By this choice, $\Gamma$ is identified with the Deck transformations group of $\tX\to X$. The action of $\Gamma$ on $\tX$ extends to an action on $T\tX$. Fix a positively oriented orthonormal frame on $(X,g)$ at $p$, and denote by $\Phi:\tilde X\to\hz$ the unique isometry mapping $p$ to $i$, such that the chosen frame is mapped to the standard frame $(\partial_x,\partial_y)$ in $T_i\hz$. The isometry $\Phi$ conjugates the action of $\Gamma$ on $\tX$ to an action of a subgroup $\rho(\Gamma)\subset \G$. To choose a spin structure on $X$ amounts to choosing a lift $\tGamma\subset \SL_2(\rz)$ of the group $\rho(\Gamma)$ under the canonical projection $\pi:\SL_2(\rz)\to\G$, i.e., such that $\pi:\tGamma\to\rho(\Gamma)$ is an isomorphism.

Recall that on negatively-curved compact manifolds, closed geodesics are in bijection with free homotopy classes of loops, which are canonically identified with conjugacy classes in the fundamental group. The function $\varepsilon$ from Definition \ref{defep} can therefore be interpreted as a class function (that is, constant on conjugacy classes) on $\Gamma$. 
\begin{remark}
One can give an alternate description of $\varepsilon$ from Definition \ref{defep} as follows: $\varepsilon(\gamma)\in\{\pm 1\}$ is the sign of the trace of $\pi^{-1}(\gamma)\in\tGamma\subset \SL_2(\rz)$.
\end{remark}
This definition also covers the case where $\gamma$ is parabolic. Since we assume $X$ to be smooth, the group $\Gamma$ does not contain elliptic elements, so the trace does not vanish.

\subsection{The spin modular group and the spin moduli space}
Diffeomorphisms of $X$ act by pull-back on the space of hyperbolic metrics and on the set of spin structures. By the homotopy invariance of the pull-back operation, the action of a diffeomorphism on the set of spin structures only depends on its isotopy class. When $X$ is a compact surface of genus $g$, we get in this way an action of the modular group $\MCG = \Diff^+(X)/\Diff^0(X)$ on the set of spin structures, a set of cardinality $2^{2g}$. Let 
$\MCG^\spin\mathrel{\unlhd}\MCG$ denote the normal subgroup fixing all spin structures on $X$. Its index in $\MCG$ is a divisor of $(2^{2g})!$, in particular it is finite. 

\begin{definition}
$\cM_g^\spin$ denotes the quotient of the Teichm\"uller space $\cT_g$ by the group $\MCG^\spin$.
\end{definition}
The set $\cM_g^\spin$ is a smooth orbifold, a branched cover of finite degree of the Riemann moduli space $\cM_g$. By definition, a smooth function on this space is the same as a $\MCG^\spin$-invariant smooth function on $\cT_g$. Since $\MCG$ preserves the K\"ahler structure on $\cT_g$, $M_g^\spin$ inherits a metric space structure with compact completion $\overline{M}_g^\spin$, whose boundary points correspond to finite-volume, complete hyperbolic surfaces. A pinching process defines in this way a path in $M_g^\spin$ towards its boundary.

Every smooth family $\{g_t\}_{t\in\cM}$ of hyperbolic metrics on $X$ indexed by some smooth manifold $\cM$, together with a fixed spin structure on $X$, defines a map $G:\cM\to M_g^\spin$ such that $g_t$ is isometric to $G(t)$, and the isometry is unique up to a diffeomorphism preserving the spin structure. In this sense, $\cM_g^\spin$ is a classifying space for families of hyperbolic structures on $X$ with a fixed spin structure.

\subsection{Generators of the fundamental group}\label{generatoripi1}
The fundamental group $\Gamma=\pi_1(X,p)$ of the compact oriented surface $X$ of genus $g$ admits a standard presentation by $2g$ generators and one relation:
\begin{equation}\label{genpi1}
 \pi_1(X, p) = \langle \alpha_1, \beta_1,... , \alpha_g, \beta_g : \ [\alpha_1,\beta_1]...[\alpha_g,\beta_g]=1   \rangle, 
\end{equation}
where $[\alpha,\beta]:=\alpha \beta\alpha^{-1} \beta^{-1} $ denotes the commutator of the elements $\alpha, \beta$ in the multiplicative group $\Gamma$. Every element $\eta \in \pi_1(X,p)$ can thus be written, albeit not uniquely, as a word in the $4g$ letters corresponding to the generators and their inverses.

\begin{definition}\label{defwordlength}
The \emph{word length} of an element $\eta \in \pi_1(X,p)$ is the minimum number of letters (generators and their inverses) needed to write $\eta$ in the above presentation of $\Gamma$. 

The word length of a conjugacy class $[\eta]$, denoted $\WL([\eta])$, is the minimum of all the word lengths of elements $\eta\in[\eta]\subset\pi_1(X,p)$.
\end{definition}

We consider a smooth family $(g_t)_{t\in \cM}$ of Riemannian metrics of curvature $-1$ on $X$ indexed by some manifold $\cM $. For instance,  $(g_t)_{t>0}$ could describe a pinching process as in Section \ref{pinchingproc}, but it could also be a section in the fibration of the set of hyperbolic metrics on $X$ over the Teichm\"uller space $\cT_g$. For each $t$, we denote by $\alpha_i^t$ and $\beta_i^t$ the unique geodesic loops based at $p$ (that generally close at an angle different from $\pi$ in $p$, so they are not closed geodesics) representing the homotopy classes $\alpha_i$ and $\beta_i$, respectively, in $\pi_1(X,p)$. For each fixed $t$, we cut the surface $X$ along these broken geodesics $\alpha_i^t, \beta_i^t$, $i=\overline{1,g}$ to get a geodesic $4g$-gon, and we consider all their lifts to $\hz$. We obtain a tiling of the hyperbolic plane $\mathbb{H}$ by congruent fundamental domains with $4g$ geodesic edges. As $t$ varies, the preimages of $p$ in $\mathbb{H}$ change smoothly, and the associated polygons deform smoothly as well.

\subsection{Smooth variations of the metric as families of representations}\label{famrep}
We can describe a smooth family of metrics $(g_t)_{t\in \cM}$ in an algebraic manner. Let $\tilde{g_t}$ be the lift of $g_t$ to the universal cover $\tilde{X}$ of $X$. If we fix $\tilde{p} \in \tilde{X}$ and an orthonormal frame at this point, then there exists an unique isometry $\Phi_t: \tilde{X} \longrightarrow \mathbb H$ which maps $\tilde{p}$ to $i$ and the fixed orthonormal frame to $\left\{ \frac{\partial}{\partial x} \big\vert_{i}, \frac{\partial}{\partial y} \big\vert_{i} \right\}$.

Let $\eta$ be an element in $\pi_1(X, p) \cong \Deck(\tilde{X},X)$, and notice that $\eta : \tilde{X} \longrightarrow \tilde{X}$ is an isometry for all the Riemannian metrics $\tilde{g_t}$. Then the map
\[ \rho_t(\eta):= \Phi_t \circ \eta \circ \Phi_t^{-1} : \mathbb H \longrightarrow \mathbb H \]
is also an isometry, thus an element in $\PSL_2(\mathbb R)$. Therefore our smooth variation of metrics $(g_t)_{t\in \cM}$ is encoded algebraically in a smooth family of faithful discrete representations 
\[ \cM\ni t\mapsto \rho_t : \pi_1(X, p) \longrightarrow \PSL_2 (\mathbb R),\]
defining a family of hyperbolic surfaces $\sqcup_{t\in \cM}(\rho_t(\pi_1(X,p)) \backslash \mathbb H)$ indexed by $\cM $.

\subsection{Trivializing the spinor bundle in time directions}\label{spinbundlesection}
 We use the method of Bourguignon and Gauduchon \cite{bgaud} to identify the frame bundles (and then the spinor bundles) of the metrics $g_t$ for different values of $t\in \cM$. Consider the generalized cylinder metric
\[  {G}:= \dd t^2 + g_t\]
on $\mathbb R_+ \times X$. The tangent bundles $(TX, g_{t_0})$ and $(TX,g_{t_1})$ for $t_0, t_1 >0$ are isometric via the parallel transport with respect to the metric ${G}$ along the vertical geodesics, i.e, the integral curves of $\partial_t$. This identifies the frame bundles $\PSO X_{t_0}$ and $\PSO X_{t_1}$ at times $t_0$ and $t_1$, inducing a canonical identification between the spin bundles $\Pspd X_{t_0}$ and $\Pspd X_{t_1}$. Moreover, outside $\gamma$, this identification also holds for $t_0=0$.

\subsection{Uniform length spectrum estimates} Denote by $l_m(t)$ the length of the $m^\text{th}$ smallest closed, oriented geodesic on $X$ with respect to the metric $g_t$, counted with multiplicity. A standard argument using a tiling of $\hz$ by Dirichlet fundamental domains and area comparisons shows that for any hyperbolic surface $X$ of genus $g$,
\[ l_m \geq \log m + \log (4g-4) - 3 \Diam X. \]
For every compact set $I\subset \cM$, $\Diam (X,g_t)_{t\in I}$ is uniformly bounded, so there exists $C=C(I)\in\rz$ such that for every $t\in I$,
\begin{equation}\label{m-alunggeod}
\begin{aligned}
l_m(t) \geq \log m -C
\end{aligned}
\end{equation}
(see for instance the proof of \cite[Lemma $9$]{rares}).

\subsection{Families of classical pseudodifferential operators}

\begin{definition}
A smooth family of pseudodifferential operators of order $k\in \cz$ on a closed manifold $X^n$ with parameters in a manifold $\cM $, acting between sections in vector bundles $\mathcal{S},\mathcal{S}'$ over $X$, is a distribution on $\cM\times X^2$ with classical conormal singularities of order $k$ at $\cM \times \Diag_X$, where $\Diag_X$ is the diagonal in $X\times X$. We denote the space of such families by
\[\mathcal C^{\infty} \lp \cM, \Psi^k(X)  \rp:=\mathcal I^k \lp \cM \times X^2, \cM \times \Diag_X \rp . \] 
\end{definition}
In order for $A_t$ to act on sections in $\mathcal S$, the distribution $A$ does not act on $C^\infty_c(\cM \times X^2)$, but rather on compactly-supported sections in the vector bundle $(\mathcal{S}')\boxtimes (\mathcal{S}^*\otimes \Omega^1(X))$. For simplicity we supress those bundles from the notation

The class of conormal distributions is stable under restriction to submanifolds transversal to the conormality locus. For every submanifold $\cM'\subset \cM$, we have thus a restriction map
$\mathcal C^{\infty} \lp \cM, \Psi^k(X)  \rp\to \mathcal C^{\infty} \lp \cM', \Psi^k(X)  \rp$. In particular for every $t\in \cM$ we have an evaluation map $\mathcal C^{\infty} \lp \cM,\Psi^k(X)  \rp\to \{t\}\times \Psi^k(X)$, so a smooth family $A\in C^{\infty} \lp \cM,\Psi^k(X) \rp$ defines a function $A:\cM \to\Psi^k(X)$. It is natural to denote by $A_t$ the evaluation of this function at $t\in \cM$.

We need the following extensions to families of standard results on pseudodifferential operators. We only sketch the proofs, they can also be seen as particular cases of the corresponding results for families of cusp-surgery operators from section \ref{cpcalc}.

\begin{lemma}[Composition]\label{composition}
Let $k_1,k_2\in\cz$ and $A_j \in \mathcal C^{\infty} \lp \cM, \Psi^{k_j}(X)\rp$, for $j=1,2$, two families of pseudodifferential operators indexed by some manifold $\cM $. Then
there exists a family $A\in \mathcal C^{\infty} \lp \cM, \Psi^{k_1+k_2}(X)\rp$ such that for all $t\in \cM$,
\[
A_t=(A_1)_t\circ (A_2)_t.
\]
\end{lemma}
\begin{proof}
The composition $A_1 A_2$ is an iteration of pull-backs, product and push-forward of conormal distributions:
\[
A_1 A_2 = {\pi_{13}}_*\lp(\pi_{12}^*A_1)(\pi_{23}^*A_2)\rp
\]
These operations commute with restriction to $\{t=\mathrm{constant\}}$, i.e., with the evaluations maps, hence 
$(A_2A_2)_t=(A_1)_t\circ (A_2)_t.$
\end{proof}

\begin{lemma}[Spectral invariance]\label{spectralinv}
If $A$ is invertible in the sense that for any $t \in \cM$ the operator $A_t$ is invertible, then the family of inverses $(A_t^{-1})_{t\in \cM}$ belongs to $\mathcal C^{\infty} \lp \cM, \Psi^{-k}(X)  \rp$.
\end{lemma}
\begin{proof}
Let $B_t \in \Psi^{-k}(X)$ be a parametrix of $A_t$, i.e., $A_tB_t = 1 + R_t$, where $R_t \in \Psi^{-\infty}(X)$. By construction, $B_t$ is a smooth family in $t$. Fix $t_0$ and let $U_{t_0}:=-R_{t_0}(A_{t_0})^{-1} \in \Psi^{-\infty}(X)$. Consider the smooth family $C_t:=B_t + U_{t_0}$ and notice that $C_{t_0}=B_{t_0}+U_{t_0}=(A_{t_0})^{-1}$. Then $A_t C_t= 1+ \tilde{R}_t$, where $\tilde{R}_t=  R_t + A_t U_{t_0} \in \Psi^{-\infty}(X) $ and $\tilde{R}_{t_0}=0$. It follows that for small enough $|t-t_0|$, $1+ \tilde{R}_t$ is invertible, and 
\[ ( 1+ \tilde{R}_t )^{-1} -1 = -\tilde{R}_t + \tilde{R}_t^2 - \tilde{R}_t^3 ...  \]
This series is convergent in the $\cC^m$ norms on $X^2$, so its sum is smooth in $t$. Since $(A_t)^{-1}=C_t (1+ \tilde{R}_t)^{-1}$, it follows by Lemma \ref{composition} that the family of inverses is also smooth in $t$, ending the proof.
\end{proof}

\begin{lemma}\label{tracesmooth}
Let $A\in\cC^\infty(\cM,\Psi^k(X))$ for some $k< -n$. Then $\Tr A \in \mathcal C^{\infty} (\cM )$, in the sense that the function $\cM \ni x \mapsto \Tr(A_x)$ is smooth. If $\cM=\cM'\times U$ and $A$ is holomorphic in $s\in U$, then $(x,s)\mapsto \Tr(A_{x,s})$ is also holomorphic in $s$.
\end{lemma}
\begin{proof}
This is a particular case of Proposition \ref{tracecp}, obtained by restricting to some $t_0>0$.
\end{proof}

\begin{lemma}\label{integr1var}
Let $\cM:=\cM_1 \times \cM_2$ be a product manifold endowed with a fixed volume form $\dd x_2$ on $\cM _2$. Then for every compact $I \subset \cM_2$ and $A\in \mathcal C^{\infty} \lp \cM, \Psi^{k}(X)\rp$,
\[x_1\mapsto \int_{I} A(x_1,x_2) \dd x_2 \in \mathcal C^{\infty} \lp \cM_1, \Psi^{k}(X) \rp.\] 
\end{lemma}
\begin{proof}
Fix a bump function $\chi\in C^\infty(X^2)$ equal to $1$ near $\Diag_X$ and supported in a tubular neighborhood $U$ of $\Diag_X$. The distribution $A$ is smooth outside $\cM \times\Diag_X$ so integrating $(1-\chi) A$ with respect to $\dd x_2$ along $I$ is smooth on $\cM _1\times X^2$. The tubular neighborhood $U$ is identified with a neighborhood of the tangent space $TX$, and by definition $\chi A$ is the fiberwise Fourier transform of a classical symbol of order $k$, say $\sigma(\chi A)(x_1,x_2,p,\xi)$. Derivatives in $x_1,p,\xi$ commute with integration in $x_2$, so the integral in $x_2$ of such a symbol $\sigma(\chi A)$ is again a classical symbol in $(x_1,p,\xi)$ of order $k$ in $\xi$. Integration in $x_2$ also commutes with Fourier transform in $\xi$, so finally we deduce that, up to a smooth distribution, $\int_{I} A(x_1,x_2) \dd x_2$ is the inverse Fourier transform of a classical symbol.
\end{proof}

\section{Key tools}

\subsection{The spin Selberg zeta function}\label{selbergderlog}
The analytic extension of the Selberg zeta function depends crucially on the trace formula. For the spin Dirac operator on a compact surface, such a trace formula was first described explicitly in \cite{BolteStiepanSelbergForDirac}.
In the smooth cofinite setting, additional terms corresponding to the continuous spectrum of the Laplacian appear in the spectral side of the trace formula, as already noticed in Selberg's original paper \cite{selberg}. For the Dirac operator on hyperbolic manifolds of finite volume, Park \cite{park} defined an odd zeta function and proved its meromorphic extension. The poles in this setting arise from the scattering matrix. In \cite{GMP} a Selberg zeta function of odd type is constructed for the spinor bundle on an odd-dimensional geometrically finite hyperbolic manifold of infinite volume and without cusps. The study of its meromorphic continuation and singularities leads to a natural eta invariant of the Dirac operator, which is then expressed in terms of the central value of the corresponding Selberg zeta function at $s=0$.

Recently, using an adapted Selberg trace formula \cite[Theorem $13$]{rares} for the spin Dirac operator $\Dd_0$ on a finite-volume hyperbolic surface $X_0$ with $k\geq 0$ cusps and admissible spin structure encoded by the class function $\varepsilon$, Stan obtained the following identity for the Selberg zeta function \eqref{szf} corresponding to $\varepsilon$, valid for $\Re s_0 > 1$ and $s\in\cz$ such that neither $(s-\frac{1}{2})^2$ nor $(s_0-\frac{1}{2})^2$ are eigenvalues of $D^2$ (see \cite[Equation $(17)$]{rares}):
\begin{equation}\label{selcomp1}
\begin{aligned}
\left(\frac{Z'}{Z}\right)(s) ={}& \frac{2s-1}{2s_0-1}   \left(\frac{Z'}{Z}\right)(s_0)  \\
{}&+ \frac{2s-1}{2} \Tr \left[  \lp \Dd_0^2 + \lp s-\tfrac{1}{2} \rp^2 \rp^{-1}   -   \lp \Dd_0^2 + \lp s_0-\tfrac{1}{2} \rp^2 \rp^{-1}   \right]  \\
{}& - \frac{\Area X}{ 4 \pi} \int_{\mathbb R} \xi \coth (\pi \xi) \lp \frac{2s-1}{\xi^2 + (s-\tfrac{1}{2})^2 }  -  \frac{2s-1}{\xi^2 + (s_0-\tfrac{1}{2})^2 }     \rp d\xi\\
{}&+k\log 2\left(1-\frac{2s-1}{2s_0-1}\right).
\end{aligned}
\end{equation}
Since the right-hand of equation \eqref{selcomp1} is meromorphic for $s\in \mathbb C$, the logarithmic derivative of the Selberg zeta function extends meromorphically to $\mathbb C$ (see \cite[Theorem 29]{rares}). The residues turn out to be positive integers, proving the \emph{holomorphic} extension of $Z$ to the complex plane. The absence of poles of $Z$ reflects the fact that the spectrum of $\Dd$ on the finite-volume surface is discrete, with our admissible choice of spin structure. 

Remark that we wrote the second term in \cite[Equation $(17)$]{rares} as half the trace of the relative resolvent of the squared Dirac operator $\mathrm{D}^2$, based on the well-known fact that the spectrum of $\Dd^2$ consists of two copies of the spectrum of $\mathrm{D}^-\mathrm{D}^+$, denoted in \cite{rares} by the sequence 
$(\xi_j^2)_{j \in \mathbb N}$. 

Equation \eqref{selcomp1} will be our essential tool in proving the smoothness of the Selberg zeta function for $\mathcal{C}^{\infty}$ variations of the hyperbolic metric on a compact surface, and also in its analysis during a pinching process towards $t=0$.

\subsection{The cusp-surgery calculus}\label{cpcalc}
In \cite{cipriana}, one of us proved that under the non-triviality hypothesis on the spin structure, the spectrum of the Dirac operator is smooth in the pinching process up to $t=0$. The core technique was the construction of a customized pseudodifferential calculus called the cusp-surgery calculus $\Psi^{*,*,*}_{\cp}(X)$, a particular case of the fibered cusp calculus with parameter of Albin-Rochon-Sher \cite{ars}. 
Through this machinery, we managed to ``glue" the family of Dirac operators $\Dd_t$ corresponding to the compact surfaces $(X,g_t)_{t > 0}$ and the limit Dirac operator $\Dd_0$ on the non-compact surface $(X \setminus \gamma,g_0)$ as a cusp-surgery operator
\begin{equation}\label{diracincalc}
    \Dd \in \Psi^{1,1,0}_{\cp}(X).   
\end{equation} 

The smoothness of the spectrum was established in  \cite[Theorem 2]{cipriana} as follows. Let $\lambda_0 \in \Spec \Dd_0$ and $\epsilon>0$ such that $[\lambda_0-\epsilon, \lambda_0 +\epsilon] \cap \Spec \Dd_0 = \{ \lambda_0 \}$. The family of spectral projectors satisfies
\begin{equation}\label{prspectralicalcul}
 \PP^t_{ [\lambda_0-\epsilon, \lambda_0 +\epsilon]} \in \Psi^{-\infty,-\infty,0}_{\cp}(X)  ,
\end{equation} 
which means precisely that there exists $t_1=t_1(\lambda_0)>0$ so that the family of Schwartz kernels corresponding to the spectral projectors $\PP^t_{ [\lambda_0-\epsilon, \lambda_0 +\epsilon]}$ for $t \in [0, t_1]$ is smooth on $[0, t_1) \times X \times X$ and vanishes rapidly at $\{ t=0 \} \times \gamma \times \gamma$.

Furthermore, in \cite[Theorem 1]{cipriana}, we proved that if $\lambda \in \mathbb R \setminus \Spec \Dd_0$, then there exists $t_0 (\lambda) >0$ such that the operators $(\Dd_t - \lambda)$ are invertible for every $ 0 \leq t \leq t_0 (\lambda)$, and the resolvents $(\Dd_t-\lambda)^{-1}$ are the evaluation at time $t$ of a cusp-surgery operator
\begin{equation}\label{rezolventecalcul}
(\Dd-\lambda)^{-1} \in \Psi^{-1,-1,0}_{\cp}(X). 
\end{equation}

We now briefly recall the construction of the cusp-surgery calculus. We emphasize that the theory is intricate, our aim in this section is only to recall the main ingredients of the construction. We refer the reader to \cite{cipriana}, and for the more general case of $\varphi$-surgery calculus to \cite{ars} for full details.

The \emph{cusp-surgery simple space} $\Xs$ is obtained by the blow-up in the sense of Melrose \cite{melrose} of the pinched curve $\gamma \times \{ 0 \}$ inside $X \times [0, \infty)$. 

Now consider the manifold with corners $X^2_{\text{b}}$, the double space of McDonald \cite{mcdonald}, obtained by blowing up the submanifold $\{t=0 \} \times \gamma_{(1)} \times \gamma_{(2)}$  inside $ [0, \infty) \times X^2$. Denote by $\ffb$ the exceptional divisor, i.e., the newly added face, and let $\tb$ be the lift of the temporal boundary $\{ t=0 \}$ by the blow-down map. 
Consider next the \emph{cusp-surgery double space} $\Xd$ defined in \cite[(5)]{cipriana} as
\begin{equation}\label{spdublu}
\left[ [0, \infty) \times X^2 ; \{ 0 \} \times \gamma_{(1)} \times \gamma_{(2)}; \{0\} \times \gamma_{(1)} ; \{0 \} \times \gamma_{(2)} ;  \lp [0, \infty) \times \Diag  \rp \cap \ffb  \right] ,
\end{equation}
where $\Diag$ is the diagonal inside $\lp X \setminus \gamma \rp^2$, and we re-denoted by $\ffb$ the lift of the face $\ffb$ from $X^2_{\text{b}}$ through the two blow-ups of the $p$-submanifolds $\{t=0\} \times \gamma_{(1)}$ and $\{t=0 \} \times \gamma_{(2)}$. Let $\ffc$ be the \emph{cusp front face} obtained through the last blow-up in \eqref{spdublu} and notice that $\ffc$ is a line bundle over $[0, \pi] \times \gamma \times \gamma$. Denote by $\tilde{\beta}: \Xd \longrightarrow [0, \infty) \times X \times X$ the blow-down map and let $\pi : X\times X \times \mathbb R_+ \longrightarrow X_{(2)}$ be the projection onto the second factor $X$. Recall from \cite[Proposition 37]{cipriana} that there exists a ``projection" (obtained by composing a true projection with a blow-down map) $\pi^2_{(1)}:\Xd \longrightarrow \Xs$ which is a \emph{b}-fibration, where the subscript in $\Xs_{(1)}$ indicates that the projection is on the simple space in the first set of variables. Of course, the corresponding result also holds true for $\pi^2_{(2)}:\Xd \longrightarrow \Xs_{(2)}$.

Let $\Ss$ be the spinorial bundle over the simple space $\Xs$ obtained using the Bourguignon-Gauduchon trick \cite{bgaud} (see Section \ref{spinbundlesection} for details on $t>0$ and \cite[Section 21]{cipriana} for more details up to $t=0$). Denote by
\[ \Ss \boxtimes \Ss^*:= \lp \pi^2_{(1)} \rp^* \Ss \otimes   \lp \pi^2_{(2)} \rp^* \Ss^*   \]
the tensor product bundle over $\Xd$ of the pull-backs of $\Ss$ and $\Ss^*$ through $\pi^2_{(1)}$ and  $\pi^2_{(2)}$, respectively. 

In order to describe the conormality ``plane", let  $\Diag:= \{ (p,p): \ p\in X \}$ and consider
\begin{equation}\label{planuldecon}
 \Delta:=  \overline{ \tilde{\beta}^{-1} \lp {\Diag} \times (0, \infty)  \rp }  \subset \Xd.
\end{equation}
Let $\Omega(X)$ be the $1$-density bundle over $X$. The space
\[ \II^{m,0,0} (\Xd):= \II^{m,0,0} \lp \Xd, \Delta , \Ss \boxtimes \Ss^* \otimes \lp  \pi_2 \circ \tilde{\beta}\rp^* \Omega \lp X \rp    \rp \]
consists of those distributions of order $m$ on $\Xd$ that
\begin{itemize}
\item[$i)$] are conormal of order $m$ to the lifted diagonal $\Delta$ defined in \eqref{planuldecon};
\item[$ii)$] are extendible across the front face $\ffc$ and the temporal boundary $\tb$ (i.e. have index sets $\mathbb N$ towards these faces);
\item[$iii)$] vanish at infinite order towards all the other faces (they have empty index sets there);
\item[$iv)$] contain a density factor pulled back from the second set of variables.
\end{itemize} 
Finally, in order to include distributions of arbitrary order of vanishing (or growth) at $\ffc$ and $\tb$, set $\II^{m, \alpha,\beta} \lp \Xd \rp := \rho_{\ffc}^{\alpha} \rho_{\tb}^{\beta}  \II^{m,0,0} (\Xd). $ 

\begin{definition}
The operators of order $(m,\alpha,\beta)$ in the \emph{cusp-surgery calculus} are defined as
\[\Psi_{\cp}^{m, \alpha, \beta} (X):=\rho_{\ffc}^{- \alpha} \rho_{\tb}^{-\beta} \Psi^{m,0,0}_{\cp}(\Xd) = \II^{m,-(\alpha+2),-\beta}(\Xd). \]
\end{definition}
They can be thought of as families of classical pseudodifferential operators on $X$ which depend smoothly on the time-parameter $t$ for $t>0$, and satisfy certain precise asymptotic properties towards a cusp pseudodifferential operator at $\{t=0\}$.

Remark that $\Psi_{\cp}^{*,*,*}(X)$  differs from $\II^{*,*,*} (\Xd)$ by a shift of $-2$ in the second index, essentially due to the density factor arising from the second set of variables.

\begin{definition}
Let $\cM$ be a smooth manifold. We say that $A=(A_x)_{x \in \cM}$ is a \emph{smooth family} of cusp-surgery operators indexed by $\cM$ if it belongs to the space of classical conormal distributions
\[  \mathcal C^{\infty} \lp \cM, \Psi^{*,*,*}_{\cp}(X) \rp := \II \lp \cM \times \Xd , \cM \times \Delta \rp.  \]
For simplicity we omit from the notation the bundles and densities. Furthermore, when $\cM=\cM'\times U$ for some complex manifold $U$, we say that the family $A$ is \emph{holomorphic in $s\in U$} if it satisfies the Cauchy-Riemann equations in $s\in U$. 
\end{definition}
In the sequel, we will use holomorphic families indexed by  an open subset $U$ of $\cz$.

\begin{lemma}
If $(A_x)_{x \in \cM} \in \mathcal C^{\infty} \lp \cM, \Psi^{m,\alpha,\beta}_{\cp}(X) \rp$ and $(B_x)_{x \in \cM} \in \mathcal C^{\infty} \lp \cM, \Psi^{m',\alpha',\beta'}_{\cp}(X) \rp$ are smooth families of cusp-surgery operators, then their composition $\lp A_x B_x \rp_{x \in \cM}$ is a smooth family of $\Psi^{m+m',\alpha+\alpha',\beta+\beta'}_{\cp}(X)$ operators. Moreover if $A$ and $B$ are holomorphic in $s\in U$, so is $AB$.
\end{lemma}
\begin{proof}
In \cite[Theorem 47]{cipriana}, we proved the composition theorem for cusp-surgery operators using adapted double and triple cusp-surgery spaces, together with the Pullback and Pushforward Theorems of Melrose (see e.g. \cite{melrose}, \cite{grieser}). Here the conclusion follows by applying the same arguments, keeping in mind that this time there is an additional parameter $x \in \cM$. The composition is bilinear in $(A,B)$, so $\frac{\partial (AB)}{\partial \overline{s}}$ vanishes whenever both $A$ and $B$ are holomorphic in $s\in U$.
\end{proof}

\begin{notation}\label{notatieco}
Let $\Omega \subset \mathbb C$ be an open set. For $m\in \nz^*\cup\{\infty\}$, we denote by $\mathcal{C}^m(\cM , \Hol(\Omega))$ the space of  $\cC^m$ functions on $\cM \times \Omega$ which are holomorphic in $s \in \Omega$, i.e., in the kernel of the Cauchy-Riemann operator ${\partial}_{\overline{s}}$.
\end{notation}

In order to control the second term in the trace formula \eqref{selcomp1}, we shall need a families version of \cite[Proposition 4]{cipriana}. Recall that for $m<-2$, classical pseudodifferential operators of order $m$ on the closed surface $X$ are trace class. For $A\in \Psi^{m,\alpha,\beta}_{\cp}(X)$, the evaluation at $t=0$, makes sense for $\beta\leq 0$. It is a cusp pseudifferential operator of order $(m,\alpha)$ which is trace class when $m<-2$ and $\alpha<0$. 
For a smooth family $A\in\cC^\infty(\cM,\Psi^{m,\alpha,\beta}_{\cp})$ with $m<-2$, $\alpha<0$ and $\beta\leq 0$,  
each operator $(A_{t,x})_{t\geq 0,x\in\cM}$ is therefore trace class, so it is natural to define $\cpTr A:[0,\infty)\times\cM\to \cz$ as the family of those traces, i.e., as the function $(t,x)\mapsto \Tr(A_{t,x})$. 

\begin{proposition}\label{tracecp}
Let $A \in \cC^\infty \lp \cM, \Psi^{m,\alpha,\beta}_{\cp}(X) \rp$ a smooth family of cusp-surgery pseudodifferential operators with $m<-2$ and $\alpha,\beta\leq 0$. 
\begin{itemize}
    \item If $\alpha - \beta \notin \mathbb{Z}$, then
\[ \cpTr A \in t^{-\alpha} \cC^\infty \lp [0,\infty)\times\cM\rp + t^{-\beta} C^\infty \lp [0,\infty)\times \cM \rp. \]
\item If $\alpha - \beta \in \mathbb{Z}$, then
\[ \cpTr A \in
t^{\min(-\alpha,-\beta)} C^\infty \lp [0,\infty)\times \cM  \rp
+ t^{\max(-\alpha,-\beta)} \log t \cdot C^\infty \lp [0,\infty)\times \cM \rp. \]
\end{itemize}
Moreover, if $\cM=\cM'\times U$ and $A$ is holomorphic in $s\in U$, then so is the trace.
\end{proposition}
\begin{proof}  
For every $t\geq 0$ and $x\in\cM$, the trace of $A_{t,x}$ equals the integral of the distributional kernel $\kappa_A$ of $A$ along the intersection of the lifted diagonal $\Delta$ with the ``slice'' $\{(t,x)\}$. (The distribution $\kappa_A$ is continuous on $\cM \times \Xd$ because of the order assumption.) Recall from \cite{cipriana} that this lifted diagonal is naturally diffeomorphic to the simple cusp-parameter space $\Xs$. We write, by the inverse Fourier transform, the restriction of $\kappa_A$ to $\cM\times \Delta$ as the integral with respect to the Liouville measure in $\xi$ of some classical symbol $a(t,x,y,\xi)$ of order $m<-2$ along the fibers of the conormal bundle. The trace $\cpTr A$ is then the push-forward of the resulting density through the product map $\cM\times \Delta\to\cM\times [0,\infty)$. The tautological map from the simple cusp-surgery space $\Xs$ to $[0,\infty)$ is known from \cite{cipriana} to be a \emph{b}-fibration.
It suffices to note that the cartesian product of a \emph{b}-fibration with a smooth manifold $\cM$ remains a \emph{b}-fibration.
The conclusion follows by applying the Pushforward theorem of Melrose as in \cite[Proposition 66]{cipriana} to this
\emph{b}-fibration. 

The holomorphy of $\cpTr(A)$ is clear, since derivatives in the parameters $s$ commute with restrictions and pushforward.
\end{proof}

\section{Smoothness of the Selberg zeta function for smooth variations of the metric}\label{Section3}
In this section we prove the smoothness of the Selberg zeta function for a smooth family of metrics $(g_t)_{t\in \cM}$ of curvature $-1$ on a compact surface $X$. To be clear, we stress that this statement does not cover pinching degenerations at $t=0$. Although the result seems rather natural, we are not aware of any reference in the literature. Caveat: our proof does not hold for general families of Selberg zeta functions \eqref{genzeta} defined by arbitrary sequences $l$ and $\varepsilon$ depending on $t\in \cM$, even when assuming strong uniform bounds on the $t$-derivatives of those sequences. As already mentioned, the proof follows by combining such bounds with the trace formula \eqref{selcomp1} for $k=0$.

A closed geodesic is called \emph{primitive} if it wraps around its geometric image in $X$ exactly once or, equivalently, if its conjugacy class is primitive in $\pi_1(X)$. We denote by $Z_t(s)$ the Selberg zeta function \eqref{szf} corresponding to the metric $g_t$ and the class function $\varepsilon$ characterizing the fixed spin structure.
From \cite[Lemma 26]{rares}, its logarithmic derivative satisfies the identity
\begin{equation}\label{derlog}
\Dl =\sum_{\mu  \ \text{primitive}} \sum_{n=1}^{\infty} \frac{l_t(\mu) \varepsilon^n(\mu) e^{-n l_t(\mu) \lp s- \tfrac{1}{2} \rp }}{ 2 \sinh \frac{n l_t (\mu)}{2}},
\end{equation}
where $\mu$ runs along all closed, oriented, primitive geodesics on $(X,g_t)$. 

\begin{notation}\label{notatieeta}
Fix $p\in X$ as the base point of the fundamental group. We regard a closed oriented primitive geodesic $\mu$ on $(X, g_t)$ as the conjugacy class $[\eta]$, for some primitive element $\eta \in \pi_1(X,p)$. If $M:=\WL ([\eta])$ is the word-length of the class $[\eta]$ (see Definition \ref{defwordlength}), we can suppose that $\eta:=a_1  ... a_M$, where $a_1,...,a_M \in \{ \alpha_j, \beta_j,\alpha_j^{-1}, \beta_j^{-1}; 1\leq j\leq g \}$ run among the generators of $\pi_1(X,p)$ and their inverses, as in Section \ref{generatoripi1}. 
\end{notation}
Notice that, for $t\in \cM$, the family of representations $(\rho_t)_{t\in\cM}$ arising from the metrics $g_t$ changes as in Section \ref{famrep}, but the fundamental group of $X$ at each ``time" $t$ remains the same. Therefore the decomposition of $\eta$ into generators of the fundamental group is independent of $t$. 

In order to prove that the logarithmic derivative of the Selberg zeta function corresponding to the metric $g_t$ is smooth in $t\in \cM$, we first need to bound the word length of a hyperbolic conjugacy class by the geodesic length of that class \emph{uniformly} for $t$ in a compact subset of $\cM$ (for the classical Milnor-\v{S}varc Lemma, see e.g. \cite[Theorem 3.13]{moira}). 

\begin{lemma}[Uniform Milnor-\v{S}varc Lemma]\label{Lwlgl}
Let $I \subset \cM$ be compact. Then there exists $A>0$ such that for all $t \in I$ and for all nontrivial conjugacy classes $[\eta]$ in $\pi_1(X,p)$ we have
\begin{equation}\label{wlgl}
\WL([\eta]) \leq A l_t([\eta]),
\end{equation}
where $l_t([\eta])$ is the length with respect to $g_t$ of the closed, oriented geodesic associated to $[\eta]$.
\end{lemma}
\begin{proof}
Let $\nu^t$ be a geodesic that closes in $p$ at an angle and has the same homotopy class as $\eta$. As in Section \ref{generatoripi1}, the universal cover $\mathbb H$ of $X$ is tiled by fundamental domains which are polygons having $4g$ edges corresponding to the geodesic generators of the fundamental group $\pi_1(X,p) = \langle \alpha_1^t, \beta_1^t,...,\alpha_g^t, \beta_g^t : [\alpha_1^t, \beta_1^t]...[\alpha_g^t, \beta_g^t] =1 \rangle$. We lift $\nu^t$ to $\tilde{\nu}^t$, a geodesic segment in $\mathbb H$ with endpoints $P^t$ and $Q^t$, preimages of the base point $p$. The geodesic $\tilde{\nu}^t$ intersects several polygons of the tiling of $\mathbb H$. Let us denote the intersection points of $\tilde{\nu}^t$ with the edges of these polygons by $A_1^t:=P^t,...,A_k^t:=Q^t$ under the following rule. If $\tilde{\nu}^t$ cuts a cluster of edges having a common point, then we denote by $A_i^t$ the first intersection of $\tilde{\nu}^t$ with the cluster, and we denote by $A_{i+1}^t$ the next intersection point which does \emph{not} belong to the cluster. Thus all the intersection points $A_1^t,...,A_k^t$ lie on disjoint edges. 

Remark that $\tilde{\nu}^t$ is homotopic to a curve $\mu^t$ in $\mathbb H$ between $A_1^t$ and $A_k^t$ consisting only of edges of the fundamental domains intersected by $\tilde{\nu}^t$. Moreover, by going ``higher" or ``lower" than $\tilde{\nu}^t$, we can  judiciously choose $\mu^t$ to contain a minimal number of such edges. 
In order to pass from an $A_i^t$ to a consecutive point $A_{i+1}^t$, we need maximum $2g$ generators, thus we notice that 
\begin{equation}\label{ecwl}
\WL ([\eta]) \leq \WL (\pr(\mu^t)) \leq 2g (k-1),
\end{equation}
where $\pr: \mathbb H \longrightarrow X$ is the canonical projection.

We claim that the distance between any two disjoint edges of the tiling is bounded below by a positive constant $d$, uniformly for $t\in I$. Indeed, let $\alpha^t$ be a generator in the tiling of $\mathbb H$, and denote by $\mathcal{D}^t$ the compact set of points which are at distance at most $1$ from $\alpha^t$. Furthermore, let $E^t$ be the set of all edges belonging to polygons of the tiling of $\mathbb H$ which have a non-empty intersection with $\mathcal{D}^t$. We claim that the number of edges in $E^t$ is bounded above uniformly for $t$ varying in the compact set $I$. Indeed, this holds true due to the fact that neither of the lengths of the generators in the tiling goes to zero in the degeneracy process as $t$ varies in $I$. Thus all the distances between edges in $E^t$ which are disjoint from $\alpha^t$ and $\alpha^t$ are bounded below by a $d_0>0$, which is again uniform for $t\in I$. Furthermore, if $\beta \notin E^t$ is an edge of the tiling of $\mathbb H$, then clearly $\dist (\alpha,\beta) > 1$, and the conclusion follows by choosing $d:=\min(d_0,1)$.

Since the points $A_i^t$ lie on disjoint edges, it follows that $\dist (A_i^t, A_{i+1}^t) \geq d$, uniformly in $t$. Therefore we get $l_t(\nu^t) = \dist(A_1^t,A_k^t) = \dist(A_1^t,A_2^t)+...+\dist(A_{k-1}^t,A_k^t) \geq d(k-1)$, and we conclude that
\begin{equation}\label{k-1c}
k-1 \leq \frac{l_t(\nu^t)}{d}. \end{equation}

Clearly the length of $\tilde{\nu}^t$ is greater than the length of the conjugacy class $[\eta]$. Moreover, for every $t\in \cM$, this conjugacy class $[\eta]$ corresponds to an unique closed oriented geodesic $\mu$ on $(X, g_t)$. Denote by $\tilde{\mu}$ the lift of $\mu$ starting at a point lying in the same fundamental domain as $A_1^t$. Using the triangle inequality, we furthermore obtain
\[ l_t([\eta]) \leq \dist(A_1^t,A_k^t) \leq l_t([\eta]) +  \dist (A_1^t, \tilde{\mu}) +  \dist (A_k^t, \tilde{\mu}) \leq l_t([\eta])+2\Diam (X, g_t).\]
Using now \eqref{ecwl} and \eqref{k-1c}, it follows that
\[ \WL ([\eta]) \leq 2g (k-1) \leq 2g \frac{l_t(\nu^t)}{d} = \frac{2g}{d}\dist(A_1^t,A_k^t)  \leq \frac{2g}{d} \lp l_t([\eta]) + 2\Diam (X, g_t) \rp, \]

Since $t$ varies in the compact $I$, it follows that the positive map $t \mapsto \sys(X,g_t)$, the length of the shortest geodesic on $X$, is uniformly bounded from below by a positive constant $c_0$. Furthermore, the function $t \mapsto \Diam (X,g_t)$ is bounded from above. Therefore there exists a $c_1$ such that
\[ \Diam (X, g_t) \leq \sys(X,g_t) \frac{\Diam (X, g_t)}{c_0} \leq l_t([\eta]) \frac{c_1}{c_0}.   \]

Thus we conclude that
\[ \WL ([\eta]) \leq l_t([\eta]) \frac{2g}{d} \lp  1+2 \frac{c_1}{c_0}  \rp , \]
and remark that the constants $d,c_0$ and $c_1$ do not depend on $[\eta]$, which ends the proof.
\end{proof}
One could try to bound the word length of a conjugacy class $[\eta]$ in $\pi_1(X,p)$ by a function of $l_t([\eta])$ uniformly in a pinching process as in Definition \ref{pinchingproc} by studying cusp excursions as in \cite{fedosovapohl}. A bound of the number of hyperbolic elements for (fixed non-degenerating) Fuchsian lattices was obtained by Eichler \cite{eichler} using Dirichlet fundamental domains.

\begin{lemma}\label{lemaboundgeod}
Let $X$ be a compact surface and $(g_t)_{t\in \cM}$ a smooth family of hyperbolic metrics on $X$. For any compact subset $I \subset \cM$ and $j$ a multi-index, there exist $K_j:= K_j(I)$ and $ B_j:=B_j(I, \alpha_1, \beta_1, ... , \alpha_{2g}, \beta_{2g})\in \rz$ (see \eqref{genpi1}) such that
for every primitive conjugacy class $[\eta]$ in $\pi_1(X,p)$
and $t\in I$,
\begin{equation}\label{boundderlung}
\frac{d^{|j|}}{dt^{j}} l_t([\eta]) \leq K_j e^{B_jl_t([\eta])}.
\end{equation} 
\end{lemma}
\begin{proof}
Recall that the length of the closed geodesic (with respect to $g_t$) determined by $[\eta]$ is given in terms of the representation $\rho_t:X\to \PSL_2\rz$ by
\begin{equation}\label{formulalungtr}
l_t \lp [\eta] \rp = 2 \mathrm{arcosh} ({\Tr \rho_t (\eta)}/{2}).
\end{equation}
Since $X$ is compact, the length of the systole of $(X,g_t)$ is  bounded below by a positive constant depending on $t$. Since $I$ is compact, this is also the case uniformly for $\lp l_t ([\eta]) \rp_{t \in I}$. By \eqref{formulalungtr}, it follows that $\lp \Tr \rho_t(\eta) \rp_{t \in I}$ is bounded below by a constant strictly greater than $2$, so for every $u>0$,
\begin{equation}\label{inegradtr}
\left(\frac{\Tr^2 \rho_t(\eta)}{4}-1\right)^{-u}\lesssim 1.
\end{equation}
For a multi-index $j$, let us consider the maximum between the $\mathcal C^{|j|}$-norms of the images of the generators in $\pi_1(X,p)$ through $\rho$:
\begin{equation}\label{Calpha}
\begin{aligned}
C_{j}:={}& \max \left\{ 1, \left\Vert \frac{\partial ^k}{\partial t^k} \rho_t (\alpha_i) \right\Vert, \left\Vert \frac{\partial ^k}{\partial t^k} \rho_t (\beta_i) \right\Vert : \ k\leq j, \ i=\overline{1,g}
\right\},
\end{aligned} 
\end{equation}
where we denoted by $\Vert \cdot \Vert$ the operator norm on $2\times 2$ matrices. The right-hand side of \eqref{formulalungtr} is independent of the choice of $\eta$ inside its conjugation class, so we write $\eta$ as a minimal word $a_1\cdot\ldots\cdot a_M$ in the generators and their inverses, where $M$ is the word length of the class $[\eta]$, as in Notation \ref{notatieeta}. By the Leibnitz rule, it follows that for every multi-index $j$, there exists a universal polynomial $P_j$ in the partial derivatives up to order $j$ of the traces of the generators, such that
\begin{equation}\label{derlungimii1}
\begin{aligned}
\frac{d^{|j|}}{dt^{j}} l_t([\eta]) 
={}&  \left( \frac{\Tr^2 \rho_t(\eta)}{4}-1\right)^{\frac{1}{2}-|j|} P_{j}\left( \frac{\partial^{|k|}}{\partial t^k} \Tr \rho_t(\eta) : k\leq j \right)  \lesssim M^{|j|}C_j^M,
\end{aligned}
\end{equation}
Finally, using the uniform Milnor-\v{S}varc lemma \ref{Lwlgl}, there exist constants $B_j$ and $K_j$ such that
\begin{align*}
    \frac{d^{|j|}}{dt^{j}} l_t([\eta]) \leq K_j e^{B_jl_t([\eta])},
\end{align*}
concluding the proof.
\end{proof}

Let $\cM$ be a smooth manifold and let $(g_t)_{t\in \cM}$ be a smooth family of hyperbolic metrics on a compact oriented surface $X$, $l_t(\mu)$ the length of the conjugacy class $[\eta]$ defining an oriented geodesic $\mu$ with respect to $g_t$, and $\varepsilon:\pi_1(X,p)\to \cz$ a bounded class function, independent of $t$, not necessarily arising from a spn structure. In the following results, we use Notation \ref{notatieco}.

\begin{lemma}\label{zetadreapta}
Let $m \in \mathbb N$. There exists a constant $\cgot_m$ which depends on $m$ and $B_1,...,B_{m}$ from Lemma \ref{lemaboundgeod} such that
\[ Z_t(s) \in \mathcal C^{m} \lp \mathcal M , \Hol \lp \{  \Re z > \cgot_m \} \rp \rp . \]
\end{lemma}
\begin{proof}
By taking the logarithm of \eqref{genzeta}
\begin{equation}\label{f(t,s)}
\log Z_t(s) = \sum_{\mu  \ \text{primitive}} \sum_{n=0}^{\infty} \log \lp 1- \varepsilon(\mu) e^{-l_t(\mu) (s+n)}  \rp, 
\end{equation}
where $\mu$ runs along all the oriented primitive closed geodesics on $(X,g_t)$. It suffices to prove the statement for the function $(t,s) \mapsto \log Z_t(s)$. 

First, we claim that $t \mapsto \log Z_t(s)$ is continuous for $s \in \{z \in \mathbb C: \Re z >2 \}$. Indeed, the sum in the RHS of \eqref{f(t,s)} has the same nature as
\begin{align*}
 \sum_{\mu  \ \text{primitive}} e^{-l_t(\mu) s} \sum_{n=0}^\infty e^{-l_t(\mu) n} ={}&  \sum_{\mu  \ \text{primitive}} e^{-l_t(\mu) s} \frac{1}{1- e^{-l_t(\mu)}}. 
\end{align*} 
Since $l_t(\mu)$ is greater than the length of the systole corresponding to $(X,g_t)$, and for $t \in I$ the map $t \mapsto \sys(X, g_t)$ is bounded below by a positive constant, $ \frac{1}{1- e^{-l_t(\mu)}}$ is bounded.
Furthermore, using \eqref{m-alunggeod}, we get
\begin{equation}\label{eq100}
\begin{aligned}
\left| \sum_{\mu  \ \text{primitive}} e^{-l_t(\mu) s} \right| \leq \sum_{n=1}^{\infty} e^{-l_n(t) \Re s} \leq \sum_{n=1}^{\infty} e^{-\Re s (\log n -C)} = e^{C \Re s} \sum_{n=1}^{\infty} n^{-\Re s},   
\end{aligned}
\end{equation}
which is absolutely convergent for $\Re s>2$, hence our claim follows. 

Let $m \geq 1$, $h>0$, and let $I \subset \mathcal M $ be a compact. Let us show that it is of class $\mathcal C^m$, for $t \in I$ in a strip $\Im s < h$ on a sufficiently right-shifted halfplane $\{\Re s > \cgot_{m,h}>h \}$. Let us first take $m-q$ derivatives in $s$ and $q$ derivatives in $t$ in the RHS of \eqref{f(t,s)}. Let us denote by $Y=\varepsilon(\mu) e^{-l_t(\mu)  (s+n)}$ and remark that
\[ \partial_s^{m-q} \partial_t^q \log (1-Y) = \frac{Y P \lp Y,l_t(\mu),l_t'(\mu),...,l_t^{(m)}(\mu),s+n \rp}{(1-Y)^m}, \]
where $P$ is a polynomial with integer coefficients which depends only on $m,q$. Let $d_0$ be the degree of $P$ in $Y$, $d_1$ be the sum of its degrees in the variables $l_t(\mu),l_t'(\mu),...,l_t^{(m)}(\mu)$, and $d_2$ its degree in $s+n$. Notice that if $q=0$, then $d_2=0$. Since $|Y|$ is uniformly bounded for all $\mu, s,n$, using Lemma \ref{lemaboundgeod}, we obtain that
\[  \left| P\lp Y,l_t(\mu),l_t'(\mu),...,l_t^{(m)}(\mu),s+n \rp \right| \lesssim (d_0+1)|s+n|^{d_2} e^{d_1 l_t(\mu) \max_{j=1}^{m}B_j}. \]

Since $l_t(\mu)$ is greater than the length of the systole corresponding to $(X,g_t)$, and for $t \in I$ the map $t \mapsto \sys(X, g_t)$ is bounded below by a positive constant, there exist a constant which bounds $\left| \frac{1}{1-\varepsilon(\mu)e^{-l_t(\mu)(s+n)}} \right|$ for $n \in \mathbb N$. 
Denote by $B:=d_1 \max_{j=1}^m B_j$, and let $h>0$. Since $\Re s + n > h$ for any $n \in \mathbb N$ we obtain the following inequalities

\begin{align*}
    \left|  \partial_s^{m-q} \partial_t^q \log (1-Y) \right| \lesssim{}&  | \Re s +n+ih|^{d_2} e^{-l_t(\mu) ( \Re s +n-B)} \\
\lesssim{}&  |\Re s +n|^{d_2} e^{-\frac{l_t(\mu) (\Re s +n-B)}{2}}  e^{-\frac{l_t(\mu) (\Re s +n-B)}{2}} \\
\lesssim{}&  e^{-l_t(\mu)\frac{\Re s+n-B}{2}}.
\end{align*}
Therefore, using again \eqref{m-alunggeod}, the sum of the term by term differentiation in the RHS of \eqref{f(t,s)} is bounded by 
\begin{align*}
 {}& \left| \sum_{\mu  \ \text{primitive}} \sum_{n=0}^{\infty} \partial_s^{m-q} \partial_t^q \log (1-Y) \right| \lesssim  \sum_{\mu  \ \text{primitive}}  e^{-l_t(\mu)  \frac{\Re s-B}{2} } \sum_{n=0}^{\infty} e^{-l_t(\mu)\frac{n}{2}} \\
 \lesssim{}&   \sum_{\mu  \ \text{primitive}}  e^{-l_t(\mu)  \tfrac{\Re s-B}{2} }
=   \sum_{n=1}^{\infty}  e^{-l_n(t) \frac{\Re s-B}{2} } \leq \sum_{n=1}^{\infty} e^{-\frac{\Re s - B}{2} \lp \log n - C  \rp} = e^{C \frac{\Re s - B}{2}} \sum_{n=1}^{\infty}  n^{-\frac{\Re s-B}{2}},
\end{align*}
and the last sum is absolutely convergent for $\tfrac{\Re s-B}{2} >2$, which is equivalent to $\Re s > 4+B$. By choosing $\cgot_m:=\max (h,  4+B)$, the claim follows. 
\end{proof}

As a corollary, for every $m\geq 0$, the logarithmic derivative of $Z_t$ from \eqref{genzeta} is also of class $\cC^m$ in $t$ and holomorphic in $s$ for real part of $s$ sufficiently large:
\begin{equation}\label{derlogsmoothdreapta}
\frac{Z_t'}{Z_t}(s) \in \mathcal C^{m} \lp \cM, \Hol \lp   \{  \Re (z) > \kappa_m \} \rp \rp .  \end{equation}

We now analyze the Zeta functions corresponding to a smooth family $(g_t)_{t\in\cM}$ of hyperbolic metrics on the compact surface $X$, indexed by some manifold of parameters $\cM$.

\begin{proposition}\label{derlogafarapoli} 
Let $U \subset \cz$ be an open set on which the functions $s \mapsto \frac{Z_t'}{Z_t}(s)$ are holomorphic for any $t \in \cM$. Then
$ \frac{Z_t'}{Z_t} (s) \in \mathcal C^{\infty} \lp \cM, \Hol(U) \rp. $
\end{proposition}
\begin{proof}
Let $m \in \mathbb N$ and let $s_0 \in \mathbb C$ such that $\Re s_0 > \kappa_m$, where $\kappa_m$ is the constant from \eqref{derlogsmoothdreapta}. Recall that the Selberg trace formula \eqref{selcomp1}, valid for $t\in \cM$ and $s\in\cz\setminus\left(\frac{1}{2} + i \Spec \Dd_{t}\right)$, expresses
$\Dl$ as a sum of three terms, since the number of cusps, $k$, equals $0$ for closed surfaces.
The first term, $\frac{2s-1}{2s_0-1}\Dlsz$, is of class $\mathcal C^{m}$ in $t$ by \eqref{derlogsmoothdreapta} and polynomial in $s$, hence belongs to $\mathcal C^{m} \lp \cM, \Hol(U) \rp$.

The third term in \eqref{selcomp1} is independent of $t$ and meromorphic in $s$, with poles in $-\mathbb{N}^*+\frac{1}{2}$ (see \cite[Lemma 28]{rares}). By hypothesis, $U$ avoids this set of poles.

The second term in \eqref{selcomp1} is the spectral contribution $(s-1/2)\Tr(\cR_t(S)-\cR_t(s_0))$, where 
\begin{equation}\label{resolvent}
\cR_t(s) = \lp \Dd_t^2 + \lp s-\tfrac{1}{2} \rp^2 \rp^{-1}\end{equation}
is the resolvent of $\Dd_t^2$ at the spectral value $\lp s-\tfrac{1}{2} \rp^2$.

By hypothesis, $U$ avoids the set $\bigcup_{t \in \cM} \lp \frac{1}{2} + i \Spec \Dd_{t} \rp $, thus the family $\Dd_t^2 + (s-\frac{1}{2})^2$ is invertible for every $(t,s)\in \cM\times U$. By Lemma \ref{spectralinv}, the family of inverses is smooth of order $-2$ and holomorphic in $s$. These resolvents have the same principal symbol for every value of the spectral parameter $s$. Their difference is therefore smooth of order $-3$, hence trace class on the surface $X$, so by Lemma \ref{tracesmooth} the second term in \eqref{selcomp1} is smooth in $(t,s)$; from the Cauchy-Riemann equations, it is also holomorphic in $s \in U$. Together, it follows that $ \frac{Z_t'}{Z_t} (s) \in \mathcal C^{m} \lp \cM, \Hol(U) \rp$. Since $m \in \mathbb N$ was arbitrary, this completes the proof.
\end{proof}

\begin{lemma}\label{lambda0}
Let $t_0\in\cM$ and $\lambda_0 \in \Spec \Dd_{t_0}$ an eigenvalue of multiplicity $m_0$. Then there exists a neighborhood $I\subset \cM$ of $t_0$ and $\delta>0$ such that $\ball \cap \Spec \Dd_t$ has exactly $m_0$ eigenvalues (counted with multiplicities) for any $t \in I$.
\end{lemma}
\begin{proof}
Since $\Spec \Dd_{t_0}$ is discrete, there exists $\delta>0$ such that the closure of the open disk $\ball$ contains just the eigenvalue $\lambda_0$. 
We claim that there exists an interval $I$ around $t_0$ such that $\ball \cap \Spec \Dd_t$ contains exactly $m_0$ eigenvalues (not necessarily distinct) for any $t \in I$. Indeed, since the spectrum of $\Dd_{t_0}$ is closed, it follows that for any $\lambda \in \mathcal C_{\delta}(\lambda_0)$ we can find a ball of radius $\epsilon>0$ such that $B(\lambda,\epsilon) \cap \Spec \Dd_{t_0}= \emptyset$. Then $\Dd_{t_0} - z$ is invertible for any $|z-\lambda| < \epsilon$, and we denote by $R_{t_0}(z) \in \Psi^{-1}(X)$ the corresponding resolvent. Notice that
\begin{align*}
R_{t_0}(z) (\Dd_t-z)  = R_{t_0}(z) \lp \Dd_{t_0} - z +\Dd_{t} - \Dd_{t_0} \rp = \Id + R_{t_0}(z) \lp \Dd_{t}  - \Dd_{t_0} \rp .
\end{align*}
But $R_{t_0}(z) : \Ll^2(X) \longrightarrow \Hh^1(X)$ is a bounded operator and $\Dd_t \xrightarrow{t \to t_0} \Dd_{t_0}$ as bounded operators from $\Hh^1(X) \longrightarrow \Ll^2(X)$, thus $\Id + R_{t_0}(z) \lp  \Dd_t-\Dd_{t_0} \rp$ is invertible, and by Lemma \ref{spectralinv}, the inverse is smooth in $t$. Therefore there exists an interval $I$ around $t_0$ such that $\Dd_t-z$ is invertible for any $|z-\lambda|<\epsilon$ and $t \in I$. By Lemma \ref{spectralinv}, the family of resolvents $R_t(z)$ belongs to $\mathcal C^{\infty} \lp I \times B_{\epsilon}(\lambda), \Psi^{-1}(X) \rp$ and is holomorphic in $z$.

Let $P^+_t$ be the spectral projector of $\Dd_t$ on $\ball$ and $P^-_t$ be the spectral projector of $\Dd_t$ on $-\ball$. If $\lambda_0=0$, clearly $P_t^+=P_t^-$. By holomorphic functional calculus, the family of spectral projectors is given by 
\[P_t^+ = \frac{1}{2 \pi i} \int_{\mathcal C_{\delta}(\lambda_0)} R_t(z) dz.\] 
Using Lemma \ref{spectralinv} and \ref{integr1var}, $P_t^+ \in \mathcal C^{\infty} \lp I, \Psi^{-\infty}(X) \rp$, thus by Lemma \ref{tracesmooth}, $t \mapsto \Tr P_t$ is smooth in $t$. Since
\[ \Tr P_t = \dim \Range P_t \in \mathbb N, \]
it follows that the dimension of $\Range P_t$ is actually locally constant for $t \in I$. By shrinking $I$ if necessary, $\ball \cap \Spec \Dd_t$ contains exactly $m_0$ eigenvalues (not necessarily distinct) for any $t \in I$. Furthermore, $\Range P_t$ forms a (trivial) vector bundle of dimension $m_0$ over the connected interval $I$ with the trivialisation given by $P_{t_0}$. 
\end{proof}

We want to study the behavior of the logarithmic derivative of the Selberg zeta function near a pole of type $\frac{1}{2}+i \lambda_0$, $\lambda_0 \in \Spec \Dd_{t_0}$. Recall that the spectrum of the Dirac operator is real and symmetric with respect to zero. 

\begin{proposition}\label{comppoli}
In the context of Lemma \ref{lambda0}, there exists a smooth family of polynomials $Q_t$ in $s$ of degree $2m_0$ such that
\[ \lp t,s \rp \mapsto \frac{Z_t'}{Z_t} (s) - \frac{d}{ds}\log Q_t(s) \in \mathcal C^{\infty} \lp  I, \Hol \lp \tfrac{1}{2} + i \ball \rp \rp.  \]
\end{proposition}
\begin{proof}

Consider $m \geq 0$. Let us denote by $P_t=P_t^++P_t^-$ the spectral projector of $\Dd_t^2$ on $\mathrm{Int} (\ball^2)$. We use equation \eqref{selcomp1} for $\Re s_0 >\kappa_m$, where $\kappa_m$ is the constant from \eqref{derlogsmoothdreapta}. The first and third terms in \eqref{selcomp1} were seen in the proof of Proposition \ref{derlogafarapoli} to belong to $\mathcal C^{m} \lp  I, \Hol \lp \frac{1}{2} + i \ball \rp \rp$. Since $P_t$ is an operator of finite rank for every $t$ we rewrite the spectral term from \eqref{selcomp1} as 
\begin{equation}\label{termenul2}
\begin{aligned}
\lp s-\tfrac{1}{2} \rp \Tr \left[  \cR_t(s)-\cR_t(s_0)   \right]={}& \lp s-\tfrac{1}{2} \rp \Tr \left[ \lp 1- P_t \rp \left[  \cR_t(s)-\cR_t(s_0)   \right] \right] \\{}&+  \lp s-\tfrac{1}{2} \rp \Tr \left[ P_t  \cR_t(s)\right]
\\{}&- \lp s-\tfrac{1}{2} \rp \Tr \left[ P_t  \cR_t(s_0)\right].
\end{aligned}
\end{equation}
Clearly $\lp 1- P_t \rp   \cR_t(s)= \lp 1-P_t \rp  \lp \Dd_t^2 + \lp s-\tfrac{1}{2} \rp^2 + P_t \rp^{-1}
$.
Since $\Dd_t^2 + \lp s-\frac{1}{2} \rp^2 + P_t$ is invertible and holomorphic in $s$ for $t\in I$ and $s\in \tfrac{1}{2}+ i B_{\delta}(\lambda_0)$, by Lemma \ref{spectralinv} its inverse belongs to $\mathcal C^{\infty} \lp I \times \lp \tfrac{1}{2}+ i B_{\delta}(\lambda_0)  \rp , \Psi^{-2}(X) \rp$ and is holomorphic in $s$. Therefore the first term in equation \eqref{termenul2} extends holomorphically for $s \in \frac{1}{2}+i \ball$, so it belongs to $\mathcal C^{\infty} \lp I, \Hol \lp \tfrac{1}{2}+ i B_{\delta}(\lambda_0)  \rp  \rp$. 

Since $\Re s_0> \kappa_m>\frac{1}{2}$, it follows that $\Dd_t^2 + \lp s_0-\tfrac{1}{2} \rp^2$ is invertible, thus by Lemma \ref{spectralinv}, the inverse family is smooth in $t$. The family of spectral projectors is also smooth in $t$, therefore the term $(s-\tfrac{1}{2}) \Tr \left[ P_t   \cR_t(s_0) \right]$ is smooth in $t$ and clearly holomorphic for $s \in \tfrac{1}{2}+i \ball$.

Furthermore, the second term $ \lp s-\tfrac{1}{2} \rp \Tr \left[ P_t \lp D_t^2 + \lp s-\tfrac{1}{2} \rp^2  \rp^{-1} \right]$ can be decomposed as
\begin{align*}
{}& \frac{1}{2i} \Tr \left[ \lp P_t^+ + P_t^- \rp  \lp \Dd_t - i \lp s-\tfrac{1}{2} \rp \rp^{-1} \right] - \frac{1}{2i}\Tr \left[ \lp P_t^+ + P_t^- \rp \lp \Dd_t + i \lp s-\tfrac{1}{2} \rp \rp^{-1}  \right]
\end{align*}
that we write as the sum of the following four terms
\begin{equation}\label{4termeni}
\begin{aligned}
    {}& \frac{1}{2i} \Tr \left[ P_t^+ \lp \Dd_t - i \lp s-\tfrac{1}{2} \rp \rp^{-1} \right], \frac{1}{2i} \Tr \left[ P_t^- \lp \Dd_t - i \lp s-\tfrac{1}{2} \rp \rp^{-1} \right] , \\
    {}&\frac{1}{2i} \Tr \left[ P_t^+ \lp \Dd_t + i \lp s-\tfrac{1}{2} \rp \rp^{-1} \right], \frac{1}{2i} \Tr \left[ P_t^- \lp \Dd_t + i \lp s-\tfrac{1}{2} \rp \rp^{-1} \right].
\end{aligned}
\end{equation}
Let us first treat the case when $\lambda_0 \neq 0$. By possibly reducing $\delta$, we may assume that
\begin{equation}\label{ballmic}
\ball \cap -\ball =\emptyset.
\end{equation} 
The first two terms in \eqref{4termeni} are smooth in $t$ and we claim that they extend holomorphically for $s \in \frac{1}{2}+i \ball$. Indeed, using \eqref{ballmic}, we obtain that $P_t^+ P_t^- =0$, hence for any $s \in \frac{1}{2}+i B_{\delta}(\lambda_0) \setminus \left\{ \tfrac{1}{2} +i \lambda_0 \right\}$ the following equality holds
\[ P_t^+ \lp \Dd_t - i \lp s-\tfrac{1}{2} \rp \rp^{-1}  = P_t^+ \lp \Dd_t + P_t^- - i \lp s-\tfrac{1}{2} \rp \rp^{-1}. \] 
Notice that the RHS is holomorphic for $s \in \frac{1}{2}+i \ball$, which proves the claim for the first term in \eqref{4termeni}. The second one $\frac{1}{2i}\Tr \left[ P_t^- \lp \Dd_t - i \lp s-\tfrac{1}{2} \rp \rp^{-1} \right]$ can be treated in a similar manner.

Now the third term in \eqref{4termeni} is given by
\begin{equation}\label{dlpol}
\begin{aligned}
\frac{1}{2i}\Tr P_t^+ \lp \Dd_t + i \lp s-\tfrac{1}{2} \rp \rp^{-1} ={}& \frac{1}{2i} \Tr_{\vert_{\Range P_t^+}} \left[ \lp P_t^+ D_t P_t^+ + i \lp s-\tfrac{1}{2} \rp P_t^+ \rp_{\vert_{\Range P_t^+}}  \right]^{-1} \\
={}& \frac{1}{2} \frac{d}{ds} \log \det \lp P_t^+ D_t P_t^+ + i \lp s-\tfrac{1}{2} \rp P_t^+ \rp_{\vert_{\Range P_t^+}}.
\end{aligned}
\end{equation}
The last equality holds true because for any matrix $A \in \mathcal M_k (\mathbb C)$, we have the equality $\Tr (A-\lambda)^{-1} = - \frac{d}{d \lambda} \log \det (A- \lambda)$, and we apply it for 
\[A = \lp P_t^+ D_t P_t^+ + i \lp s-\tfrac{1}{2} \rp P_t^+ \rp_{\vert_{\Range P_t^+} } \in \End \lp \Range P_t^+ \rp.\] 
We have a similar argument for the last term $ \tfrac{1}{2i}\Tr \left[ P_t^- \lp \Dd_t + i \lp s-\tfrac{1}{2} \rp \rp^{-1} \right]$ in \eqref{4termeni}, hence for the case when $\lambda_0 \neq 0$, the sum of the last two terms in \eqref{4termeni} is $\frac{1}{2}$ multiplied by the logarithmic derivative of the polynomial 
\[Q_t(s):=  \det \lp P_t^+ D_t P_t^+ + i \lp s-\tfrac{1}{2} \rp P_t^+ \rp_{\vert_{\Range P_t^+}}  \cdot \det \lp P_t^- D_t P_t^- + i \lp s-\tfrac{1}{2} \rp P_t^- \rp_{\vert_{\Range P_t^-}} \] 
of degree at most $2m_0$ in $s$ with coefficients which are smooth functions of $t$. 

Suppose now that $\lambda_0=0$, thus $P_t^+=P_t^-$. Then instead of the four terms in \eqref{4termeni}, we have:
\begin{equation}\label{2termeni}
\begin{aligned}
    {}& \frac{1}{i} \Tr \left[ P_t^+ \lp \Dd_t - i \lp s-\tfrac{1}{2} \rp \rp^{-1} \right], && \frac{1}{i} \Tr \left[ P_t^+ \lp \Dd_t + i \lp s-\tfrac{1}{2} \rp \rp^{-1} \right],
\end{aligned}
\end{equation}
and both terms can be treated as in \eqref{dlpol}. More precisely, their sum is the logarithmic derivative of 
\begin{align*}
Q_t^0(s)={}&  \det \lp  P_t^+ D_t P_t^+ - i \lp s-\tfrac{1}{2} \rp P_t^+  \rp_{\vert_{\Range P_t^+}} \cdot \det \lp P_t^+ D_t P_t^+ + i \lp s-\tfrac{1}{2} \rp P_t^+ \rp_{\vert_{\Range P_t^+}} . 
\end{align*} 
Therefore we proved that if we subtract from the logarithmic derivative of the Selberg zeta function the logarithmic derivative of $Q_t(s)$, or $Q_t^0(s)$ if $\lambda_0=0$, we obtain an element of $\mathcal C^{m} \lp I, \Hol \lp \frac{1}{2} + i \ball \rp \rp$. Since $m$ was chosen arbitrarily, the conclusion follows.
\end{proof}

\begin{theorem}\label{smdepz}
For every smooth family of hyperbolic metrics on a compact oriented surface $X$ and for every spin structure, 
the family of Selberg zeta functions $\{Z_t(s)\}_{t\in\cM}$ defined in \eqref{szf} belongs to $\mathcal C^{\infty} \lp \cM, \Hol (\mathbb C) \rp$. 
\end{theorem}
\begin{proof}
Let $m \in \mathbb N$, $t_0>0$, and $s_0 \in \mathbb C$ such that $\Re s_0 > \cgot_m$, where $\cgot_m$ is the constant from Lemma \ref{zetadreapta}. Furthermore, let $R>0$ so that $s_0 \in B_R(0)$. Denote by $\mathcal P$ the set of poles of the function $z \mapsto \frac{Z_{t_0}'}{Z_{t_0}}(z)$. We will prove the statement using different strategies for the region $\mathbb C \setminus \mathcal P$, near poles of the form $\frac{1}{2} + i \lambda_0$ with $\lambda_0 \in \Spec \Dd_{t_0}$, and near poles of the form $-n+\frac{1}{2}$, where $n \in \mathbb N$. Let us first consider $s \in \mathbb C \setminus \mathcal P$. Then there exists $\delta>0$ such that $s \notin B_{\delta}(p)$, for any pole $p \in \mathcal P$. By \cite[Theorem 1.2]{pfaffle}, there exists an interval $I$ around $t_0$ such that for any $t \in I$, the poles of the function $z \mapsto \frac{Z_t'}{Z_t}(z)$ are included in $\cup_{p \in \mathcal P} B_{\delta}(p)$. In particular, $s$ is not a pole of $z \mapsto \frac{Z_t'}{Z_t}(z)$ for any $t \in I$.  

Let \(\eta(s):[0,1]\to\mathbb{C}\) be the path from $s_0$ to $s$ consisting of straight line segments together, if needed, with trigonometric arcs along the circles of radius $\delta$ centered at the points in $\mathcal P$. Then $\eta(s)$ avoids the poles of $z \mapsto \frac{Z_{t}'}{Z_{t}}(z)$ for any $t \in I$. Thus using Proposition \ref{derlogafarapoli} and the fact that the path $\eta(s) \subset B_R(0)$ has finite length, we get the well-defined function
\[  s \mapsto \int_{\eta(s)} \frac{Z_t'(w)}{Z_t(w)}dw . \]

For any $t \in I$, the residues of the function $z \mapsto \frac{Z_t'}{Z_t}(z)$ are integers (see e.g. \cite[Theorem 29]{rares}), thus if $\mu$ is another path between $s_0$ and $s$, it follows that 
\[\int_{\mu} \frac{Z_t'(w)}{Z_t(w)}dw-\int_{\eta(s)} \frac{Z_t'(w)}{Z_t(w)}dw \in 2 \pi i \mathbb Z .\]

Therefore the function $s \mapsto e^{\int_{\eta(s)} \frac{Z_t'}{Z_t} (w) dw}$ is well-defined (it does not depend on the chosen path between $s_0$ and $s$, so we can compute it using $\eta(s)$), and by Lemma \ref{derlogafarapoli}, it belongs to $\mathcal C^{\infty}  \lp I, \Hol \lp \mathbb C \setminus \cup_{p \in \mathcal P} B_{\delta}(p) \rp \rp$. Since $\Re s_0 > \cgot_m$, using Lemma \ref{zetadreapta}, it follows that $Z_t(s_0)$ is of class $\mathcal C^m$ in $t$, thus we conclude that
\[  Z_t(s)=Z_t(s_0) e^{\int_{\eta(s)} \frac{Z_t'}{Z_t} (w) dw}  \in \mathcal C^m \lp I, \Hol (\mathbb C \setminus \cup_{p \in \mathcal P} B_{\delta}(p)) \rp.  \]
Since $m$ was arbitrarily chosen, we obtain $ Z_t(s) \in \mathcal C^{\infty} \lp I, \Hol (\mathbb C \setminus  \cup_{p \in \mathcal P} B_{\delta}(p)) \rp$.

Now let us study the behavior of the family of Selberg zeta functions at a pole of the form $\tfrac{1}{2}+i\lambda_0 \in \mathcal P$, $\lambda_0 \in \Spec \Dd_{t_0}$. Shrinking $\delta$ if necessary, by Proposition \ref{comppoli} we have that
\begin{equation}\label{ecuatiedeintexp}
f(t,s) := \frac{Z_t'}{Z_t} (s) - \frac{d}{ds} \log Q_t(s) \in \mathcal C^{\infty} \lp I, \Hol \lp \tfrac{1}{2}+ i B_{\delta}(\lambda_0 ) \rp \rp.  
\end{equation} 

Notice that, for any $t\in I$, the function $s \mapsto f(t,s)$ is holomorphic on $B_R(0) \setminus \cup_{p \in \mathcal P \setminus \left\{ \tfrac{1}{2} + i \lambda_0 \right\} } B_{\delta}(p)$. By integrating and exponentiating \eqref{ecuatiedeintexp}, we get
\[ h (t,s): = e^{\int_{\eta(s)} f(t,w)dw} \in \mathcal C^{\infty}\lp  I, \Hol \lp \tfrac{1}{2}+ i B_{\delta}(\lambda_0) \rp \rp,  \]
and moreover, $ h(t,s) =  \frac{Z_t(s)}{Z_t(s_0)} \cdot \frac{Q_t(s_0)}{Q_t(s)}$ on $\frac{1}{2}+ i B_{\delta}(\lambda_0) \setminus \left\{ \tfrac{1}{2} + i \lambda_0 \right\}$. Therefore
\[ Z_t(s) = Z_t(s_0) h(t,s) \frac{Q_t(s)}{Q_t(s_0)} \]
on $\frac{1}{2}+ i B_{\delta}(\lambda_0) \setminus \left\{ \frac{1}{2} + i \lambda_0 \right\}$. Both terms are regular at $s=\frac{1}{2} + i \lambda_0$, and by continuity their values at $s=\frac{1}{2} + i \lambda_0$ coincide. Remark that $Z_t(s_0)$ is of class $\mathcal C^m$ in $t$; that $h(t,s) \in \mathcal C^{\infty} \lp I, \Hol \lp \frac{1}{2} + i B_{\delta}(\lambda_0)  \rp \rp$; and, by Proposition \ref{comppoli}, $Q_t(s_0)$ does not vanish. Therefore 
\[ Z_t(s) \in \mathcal C^{m} \lp I, \Hol \lp \tfrac{1}{2} + i B_{\delta}(\lambda_0)  \rp \rp.\]
Since $m$ was arbitrary, the conclusion also follows around the poles of the form $\frac{1}{2} + i \lambda_0$, $\lambda_0 \in \Spec \Dd_{t_0}$.

Finally, we study the behavior of the family of Selberg zeta functions at a pole of the form $-n+\frac{1}{2} \in \mathcal P$ with $n \in \mathbb N$. By equation \eqref{selcomp1} (and \cite[Lemma 28]{rares}), we have that 
\begin{align*}
\frac{Z'_t}{Z_t}(s) +  \frac{Z'_t}{Z_t}(1-s) ={}& -\frac{\Area X}{4 \pi} (2s-1) \lp -\pi \tan (\pi s) + \pi \tan (\pi-\pi s) \rp    \\
={}& \Area X \lp s- \tfrac{1}{2} \rp \tan (\pi s)=:\varphi(s).
\end{align*} 
Notice that $s \mapsto \varphi(s)$ is a meromorphic function with simple poles at $\frac{1}{2} + \mathbb Z$, and residues given by $\Res_{s=\frac{1}{2}+k} \varphi(s) = -4(g-1)k$, $k \in \mathbb Z$. Thus for $k \leq -1$, the residues are positive integers. By integrating, we get
\[ \log \frac{Z_t(s)}{Z_t(1-s)}  = \int_{1/2}^{s} \varphi(w)dw, \]
which is well-defined modulo $2 \pi i$. Furthermore, by exponentiating, we obtain that
\[ \frac{Z_t(s)}{Z_t(1-s)}=e^{\log  \frac{Z_t(s)}{Z_t(1-s)}}  =e^{\int_{1/2}^{s} \varphi(w)dw},  \]
where the RHS is a meromorphic function with singularities (poles and zeros) at the poles of $\varphi$. Therefore we get the following equality of meromorphic functions:
\begin{equation}
Z_t(s)= Z_t(1-s) e^{\int_{1/2}^{s} \varphi(w)dw}.
\end{equation}
Notice that $Z_t(1-s)\in \mathcal C^{\infty} \lp I, \Hol \lp B_{\delta} \lp-n+\frac{1}{2} \rp \rp  \rp$, and $e^{\int_{1/2}^{s} \varphi(w)dw}$ is independent of $t$ and has a zero of order $\Res_{s=-n+\frac{1}{2}} \varphi$ at $-n+\frac{1}{2}$. Therefore 
\[ Z_t(s)  \in  \mathcal C^{\infty} \lp I, \Hol \lp B_{\delta} \lp -n+\tfrac{1}{2} \rp \rp  \rp, \]
concluding the proof.
\end{proof}

Theorem \ref{thmain1} follows directly from this result and from the existence of local \emph{slices} i.e., smooth sections, in the $\Diff^0$-principal fibration of the manifold of hyperbolic metrics above $\cT_g$.

Notice that exactly the same proof shows that the Selberg Zeta function for the Laplacian is smooth on the Riemann moduli space $\cM_g$. 

\section{Behaviour of the logarithmic derivative of the Selberg zeta function in the pinching process}
Suppose that our spin compact oriented hyperbolic surface $X$ goes through a pinching process of $k$ mutually disjoint simple closed geodesics as in Definition \ref{pinchingproc}, thus the limit surface corresponding to $t=0$ has $2k$ cusps.

\begin{proposition}\label{atomi}
Let $\lambda \in \Spec \Dd_0$ and $\epsilon>0$ such that 
$ [\lambda-\epsilon,\lambda+\epsilon] \cap \Spec \Dd_{0}=\{ \lambda \}.  $
Let $P^t_\lambda$ be the spectral projector of $\Dd_t$ on $[\lambda-\epsilon,\lambda+\epsilon]$. Then there exists $t_{\lambda}>0$ so that the rational function in $s$
\[ Q^{\lambda}_{t}(s) :=  \frac{ \det \left[ P^t_{\lambda} \Dd_t^2 P^t_{\lambda}+ \lp s-\tfrac{1}{2}\rp^2 P^t_{\lambda} \right]_{\vert_{\Range P^t_{\lambda}}}}{ \det \left[ P^0_\lambda \Dd_0^2 P^0_\lambda+ \lp s-\tfrac{1}{2}\rp^2 P^0_\lambda \right]_{\vert_{\Range P^0_\lambda}}}  \]
has smooth coefficients in $t \in [0, t_{\lambda})$.
\end{proposition}
\begin{proof}
The conclusion follows directly from the smooth dependence of the spectral projections on the parameter $t$ (see \eqref{prspectralicalcul}) which was established in \cite{cipriana}. 
\end{proof}

For $R>0$ and $t_R$ the minimum of $t_{\lambda}$ for $\lambda \in \Spec \Dd_{0}$ and $|\tfrac{1}{2}+i \lambda | <R$, define 
\begin{equation}\label{QtR}
 Q_{t,R}(s)=\prod_{\substack{\lambda \in \Spec \Dd_0 \\
\tfrac{1}{2}+i \lambda \in B_R(0)}} Q^{\lambda}_{t}(s). 
\end{equation}
It is a rational function of degree $0$ with smooth coefficients in $t \in [0, t_R)$. Furthermore, when $t=0$, $Q_{0,R}(s)$ equals $1$. 

We denote by $\mathcal C_R(0)$ the circle of radius $R$ centered at the origin, and by $B_R(0)$ the corresponding open disk of radius $R$ centered at $0$.

\begin{proof}[Proof of Theorem \ref{mainth}]
By \cite[Theorem 34]{rares}, we know that
\begin{equation}\label{limitarares}
\lim_{t \to 0} \Dl + 2k \log 2 = \lp \frac{Z'_0}{Z_0}\rp(s), 
\end{equation}
uniformly on compacts sets in $\mathbb C$ avoiding the poles of the limit function. Let us study $\frac{d}{ds}\left[\frac{1}{2s-1} \lp \Dl + 2k \log 2 \rp \right]$ using the Selberg trace formula from Section \ref{selbergderlog} for the compact case. 

Add $2k\log 2$ to the functional equation \eqref{selcomp1} and then divide by $2s-1$. We get for $t>0$ and $\Re s_0 > 1$ the identity
\begin{align*}
\frac{1}{2s-1} \left[  2k\log 2 + \Dl  \right]{}&= \frac{2k \log 2}{2s-1} + \frac{1}{2s_0-1} \Dlsz  \\ {}&+ \frac{1}{2} \Tr \left[  \lp \Dd_t^2 + \lp s-\tfrac{1}{2} \rp^2 \rp^{-1}   -   \lp \Dd_t^2 + \lp s_0-\tfrac{1}{2} \rp^2 \rp^{-1}   \right]  \\
{}& - \frac{\Area X}{ 4 \pi} \int_{\mathbb R} \xi \coth (\pi \xi) \lp \frac{1}{\xi^2 + (s-\tfrac{1}{2})^2 }  -  \frac{1}{\xi^2 + (s_0-\tfrac{1}{2})^2 }     \rp d\xi .
\end{align*}

We now differentiate with respect to $s$:
\begin{equation}\label{derzeta1}
\begin{aligned}
\lefteqn{\frac{d}{ds}\left[\frac{1}{2s-1} \left[\lp 2k \log 2 + \frac{Z'_t}{Z_t}  \rp (s) \right] \right] =}\\
={}&  \frac{-4k \log 2}{(2s-1)^2} 
+ \frac{1}{2} \frac{d}{ds}  \Tr \left[  \lp \Dd_t^2 + \lp s-\tfrac{1}{2} \rp^2 \rp^{-1}   -   \lp \Dd_t^2 + \lp s_0-\tfrac{1}{2} \rp^2 \rp^{-1}   \right] 
-\frac{d}{ds} F(s), 
\end{aligned}
\end{equation}
where $F(s)$ is a meromorphic function independent of $t$:
\begin{align*}
F(s):={}& \frac{\Area X}{ 4 \pi} \int_{\mathbb R} \xi \coth (\pi \xi) \lp \frac{1}{\xi^2 + (s-\tfrac{1}{2})^2 }  -  \frac{1}{\xi^2 + (s_0-\tfrac{1}{2})^2 }     \rp d\xi  \\
={}& - \frac{\Area X}{2 \pi} \left[ \frac{1}{2 s_0-1} - \frac{1}{2s-1} + \sum_{n=1}^{\infty} \lp \frac{1}{s_0-\tfrac{1}{2} +n} - \frac{1}{s-\tfrac{1}{2} +n}   \rp \right],
\end{align*} 
and the last equality holds by \cite[Lemma 28]{rares}.
Remark that the second term in the RHS of \eqref{derzeta1} is given by:
\begin{align*}
 -  \lp s- \tfrac{1}{2} \rp \Tr \left[ \Dd_t^2 + \lp s - \tfrac{1}{2} \rp^2  \right]^{-2} =- \lp s-\tfrac{1}{2} \rp \Tr \Rr_{\Dd_t^2}^2 \lp - \lp s-\tfrac{1}{2} \rp^2 \rp.
\end{align*}
Therefore equation \eqref{derzeta1} becomes
\begin{equation}\label{derzetafinal}
\begin{aligned}
\frac{d}{ds}\left[\frac{1}{2s-1} \left[ 2k \log 2 +\Dl \right] \right] ={}&  \frac{-4k \log 2}{(2s-1)^2} - \lp s-\tfrac{1}{2} \rp \Tr \Rr_{\Dd_t^2}^2 \lp - \lp s-\tfrac{1}{2} \rp^2 \rp \\
{}&- \frac{d}{ds} F(s).
\end{aligned}
\end{equation}

Now let us study how the operations $\frac{d}{ds}$ and $\frac{1}{2s-1}$ act on the RHS of the limit \eqref{limitarares}. By \cite[Equation $(17)$]{rares}, we have
\begin{align*}
\lp \frac{Z'_0}{Z_0}  \rp (s) ={}&  2k \log 2 \lp 1- \frac{2s-1}{2s_0-1} \rp+ \frac{2s-1}{2s_0-1} \lp  \frac{Z'_0}{Z_0} \rp(s_0) \\ {}&+ \frac{2s-1}{2} \Tr \left[  \lp \Dd_0^2 + \lp s-\tfrac{1}{2} \rp^2 \rp^{-1}   -   \lp \Dd_0^2 + \lp s_0-\tfrac{1}{2} \rp^2 \rp^{-1}   \right] \\
{}& - \frac{\Area X}{ 4 \pi} \int_{\mathbb R} \xi \coth (\pi \xi) \lp \frac{2s-1}{\xi^2 + (s-\tfrac{1}{2})^2 }  -  \frac{2s-1}{\xi^2 + (s_0-\tfrac{1}{2})^2 }     \rp d\xi.
\end{align*}

Then using similar computations as before, we get:
\begin{equation}\label{derzeta2}
\begin{aligned}
\frac{d}{ds} \left[ \frac{1}{2s-1} \lp \frac{Z'_0}{Z_0}  \rp(s) \right]
={}& \frac{-4k \log 2}{(2s-1)^2}  - \lp s-\tfrac{1}{2} \rp\Tr \Rr_{\Dd_0^2}^2 \lp - \lp s-\tfrac{1}{2} \rp^2 \rp -\frac{d}{ds}F(s).
\end{aligned}
\end{equation}

By \eqref{derzetafinal} and \eqref{derzeta2}, we conclude that
\begin{equation}\label{concl}
\begin{aligned}
\frac{d}{ds}\left[\frac{1}{2s-1} \left[\Dl + 2k \log 2 - \lp \frac{Z'_0}{Z_0}  \rp (s) \right] \right] ={}& \lp \tfrac{1}{2}- s \rp \Tr \lp \Rr^2_{\Dd_t^2}  \rp \lp  - \lp s-\tfrac{1}{2} \rp^2 \rp \\
-{}& \lp \tfrac{1}{2}- s \rp \Tr \lp  \Rr^2_{\Dd_0^2} \rp  \lp  - \lp s-\tfrac{1}{2} \rp^2 \rp .
\end{aligned}
\end{equation}

By hypothesis, $\mathcal C_R(0) \cap \lp \tfrac{1}{2} + i \Spec \Dd_0 \rp = \emptyset$. Hence, by \cite[Theorem 1.2]{pfaffle}, there exists $t_0>0$ such that for any $t \leq t_0$, the circle of radius $R$ centered at $0$ does not intersect $\tfrac{1}{2}+ i \Spec \Dd_t$. Let $P_R^t$ be the spectral projector of $\Dd_t$ corresponding to $\{ \lambda \in \Spec \Dd_t: \tfrac{1}{2}  + i \lambda \in B_R(0) \}$.

We split the resolvents in \eqref{concl} by projecting near $\tfrac{1}{2}+i \Spec \Dd_t$ for each $t \geq 0$, and away from it:
\begin{equation}\label{proiectiipezero}
\begin{aligned}
{}& \lp \Rr^2_{\Dd_t^2}  \rp \lp  - \lp s-\tfrac{1}{2} \rp^2 \rp =  P^t_R  \Rr^2_{\Dd_t^2}   \lp  - \lp s-\tfrac{1}{2} \rp^2 \rp  +  (1-P^t_R)  \Rr^2_{\Dd_t^2} \lp  - \lp s-\tfrac{1}{2} \rp^2 \rp.
\end{aligned}
\end{equation}
Let us study the term $(1-P^t_R)  \Rr^2_{\Dd_t^2} \lp  - \lp s-\tfrac{1}{2} \rp^2 \rp$ in \eqref{proiectiipezero}. Recall from Section \ref{cpcalc} that we merged the correspondent family of Dirac operators $(\Dd_t)_{t>0}$ on the hyperbolic surfaces $(X,g_t)$, together with the Dirac operator $\Dd_0$ on the limit non-compact surface $(X \setminus \gamma,g_0)$ to obtain a cusp-surgery differential operator $\Dd$ of orders $(1,1,0)$. It follows that $\Dd^2 \in \Psi_{\cp}^{2,2,0}(X)$.
Moreover, by \cite[Theorem $1$]{cipriana}, we have that $\lp P^t_{R} \rp_{t \geq 0}=: P_{R} \in \Psi^{-\infty,-\infty,0}_{\cp}(X)$, thus $1-P_{R} \in \Psi_{\cp}^{0,0,0}(X)$, and therefore we get for $t \in [0, t_0)$, 
\[ (1-P_{R})  \Rr^2_{\Dd^2} \lp  - \lp s-\tfrac{1}{2} \rp^2 \rp \in \Hol \lp B_R(0)),  \Psi_{\cp}^{-4,-4,0}(X) \rp.  \]

By Proposition \ref{tracecp}, it follows that 
\begin{equation}\label{infotezacip1}
\begin{aligned}
{}^{\cp} \Tr  \lp 1-P_{R} \rp \Rr_{\Dd^2}^2 \lp - \lp s-\tfrac{1}{2} \rp^2  \rp {}&\in  \mathcal C^{\infty} \lp [0, t_0), \Hol \lp B_R(0) \rp \rp \\
{}&+ t^4 \log t  \mathcal C^{\infty} \lp [0, t_0), \Hol \lp B_R(0)  \rp \rp.
\end{aligned}
\end{equation}

We will now address the term $P^t_R  \Rr^2_{\Dd_t^2} \lp  - \lp s-\tfrac{1}{2} \rp^2 \rp$ in \eqref{proiectiipezero}, which is a meromorphic family of finite rank operators. Since the spectral projector $P^t_R$ and the resolvent commute, it is also equal to $P^t_R  \Rr^2_{\Dd_t^2} \lp  - \lp s-\tfrac{1}{2} \rp^2 \rp  P^t_R$. Recall that for any matrix $A \in \mathcal M_k (\mathbb C)$, we have $\Tr \lp A-\lambda \rp^{-1} = -\frac{d}{d \lambda} \log \det (A-\lambda)$, and furthermore
\[ \Tr (A-\lambda)^{-2}= \Tr \left[ \frac{d}{d \lambda} (A-\lambda)^{-1}  \right]   =- \frac{d^2}{d \lambda^2} \log \det  (A-\lambda). \]
We apply the above equality for $P_R^t R_{\Dd_t}^2 \lp  - \lp s-\tfrac{1}{2} \rp^2 \rp P^t_R  \in \End \lp \Range P_R^t \rp$, and we obtain 
\begin{equation}
\begin{aligned}
{}&\Tr P^t_R  \left[ \Dd_t^2+ \lp s-\tfrac{1}{2}\rp^2  \right]^{-2}  P^t_R=\Tr_{\vert_{\Range P^t_R}} \left[ \left[ P^t_r \Dd_t^2 P^t_R+ \lp s-\tfrac{1}{2}\rp^2 P^t_R \right]_{\vert_{\Range P^t_R}} \right]^{-2} \\
{}&=- \lp \frac{1}{2s-1} \frac{d}{ds} \rp^2  \log \det \left[ P^t_R \Dd_t^2 P^t_R+ \lp s-\tfrac{1}{2}\rp^2 P^t_R \right]_{\vert_{\Range P^t_R}},  
\end{aligned}
\end{equation}
where notice that $\det \left[ P^t_R \Dd_t^2 P^t_R+ \lp s-\tfrac{1}{2}\rp^2 P^t_R \right]_{\vert_{\Range P^t_R}}$ is a polynomial in $s$ of degree at most the sum of the multiplicities of the eigenvalues $\{ \lambda \in \Spec \Dd_0: \tfrac{1}{2}+i \lambda \in B_R(0)\}$, with smooth coefficients in $t$. Thus $\Tr  P^t_R  \Rr^2_{\Dd_t^2}   \lp  - \lp s-\tfrac{1}{2} \rp^2 \rp - \Tr  P^0_R  \Rr^2_{\Dd_0^2}   \lp  - \lp s-\tfrac{1}{2} \rp^2 \rp$ contributes to \eqref{concl} with
\begin{equation}\label{contribP0}
- \lp \frac{1}{2s-1} \frac{d}{ds} \rp^2  \left[ \log \frac{ \det \left[ P^t_R \Dd_t^2 P^t_R+ \lp s-\tfrac{1}{2}\rp^2 P^t_R \right]_{\vert_{\Range P^t_R}}}{ \det \left[ P^0_R \Dd_0^2 P^0_R+ \lp s-\tfrac{1}{2}\rp^2 P^0_R \right]_{\vert_{\Range P^0_R}}}   \right] ,
\end{equation}
which is a differential operator in $s$ applied to the logarithm of a rational function of $s$ with smooth coefficients in $t \in [0, t_0)$. Moreover, using Proposition \ref{atomi}, we can split $P^t_R$ as a direct sum of spectral projectors concentrated at each point $\lambda \in \Spec \Dd_0$. Indeed, by taking a smaller $t_0$ if needed (the minimum between $t_0$ and all the values $t_\lambda$ introduced for $\lambda \in \Spec \Dd_0$ in Proposition \ref{atomi}), we can write 
\[P^t_R=\sum_{{\substack{\lambda \in \Spec \Dd_0 \\
\tfrac{1}{2}+i \lambda \in B_R(0)}}} P^t_{\lambda},\]
and thus
\begin{equation}\label{contribbb}
\begin{aligned}
\frac{ \det \left[ P^t_R \Dd_t^2 P^t_R+ \lp s-\tfrac{1}{2}\rp^2 P^t_R \right]_{\vert_{\Range P^t_R}}}{ \det \left[ P^0_R \Dd_0^2 P^0_R+ \lp s-\tfrac{1}{2}\rp^2 P^0_R \right]_{\vert_{\Range P^0_R}}}  = \prod_{\substack{\lambda \in \Spec \Dd_0 \\
\tfrac{1}{2}+i \lambda \in B_R(0)}} Q^t_{\lambda}
(s) \overset{\text{\eqref{QtR}}}{=} Q_{t,R}(s).
\end{aligned}
\end{equation}

By combining \eqref{concl}, \eqref{infotezacip1}, \eqref{contribP0}, \eqref{contribbb}, and the fact that the RHS of  \eqref{concl} vanishes at $t=0$ the conclusion follows.
\end{proof}
Let us remark that choosing $t_0$ in a similar way as before, on a relatively compact set $U$ whose closure avoids the set $\tfrac{1}{2} + i \Spec \Dd_0$, we obtain 
\begin{equation}\label{comportamentdeppoli}
\begin{aligned}
\frac{d}{ds}  \frac{1}{2s-1} \left[ \lp \frac{Z_t'}{Z_t}  \rp(s) -  \lp \frac{Z_0'}{Z_0} \rp (s) +  2k \log 2 \right] {}&\in  t  \mathcal C^{\infty} \lp [0,t_0), \Hol \lp U \rp \rp \\
{}&+ t^4 \log t  \mathcal C^{\infty} \lp [0,t_0), \Hol \lp U \rp \rp.
\end{aligned}
\end{equation}
Moreover, if we localize near a pole of the limit logarithmic derivative of the Selberg zeta function of type $\tfrac{1}{2}+ i \lambda$, $\lambda \in \Spec \Dd_0$, then we get
\begin{align*}
{}&\frac{d}{ds} \frac{1}{2s-1} \left[\Dl - \lp \frac{Z'_0}{Z_0}  \rp (s) + 2k \log 2  - \frac{1}{2}   \frac{d}{ds}  \log Q_t^{\lambda}(s) \right]  \\
{}&\in t \mathcal C^{\infty} \lp [0, t_0), \Hol \lp B_{\epsilon} ( \tfrac{1}{2} + i \lambda) \rp \rp 
+ t^4 \log t  \mathcal C^{\infty} \lp [0, t_0), \Hol \lp B_{\epsilon}(\tfrac{1}{2}+i \lambda) \rp \rp,
\end{align*}
where $\epsilon$, $Q_t^{\lambda}(s)$ and $t_0:=t_{\lambda}$ are provided by Proposition \ref{atomi}.

\subsection{Coefficients at $s=1/2$ of the logarithmic derivative of the Selberg zeta function in the pinching process}
We specialize the previous result in a neighborhood of the special value $\frac{1}{2}$. By \cite[Theorem 29]{rares}, $0 \in \Spec \Dd_t^2$ if and only if $1/2$ is a pole of $\frac{Z_t'}{Z_t}$, the corresponding logarithmic derivative of the Selberg zeta function. 
Suppose from now on that $0 \notin \Spec \Dd_0^2$. Then using \cite[Theorem 1.2]{pfaffle}, there exists $t_0>0$ such that $0$ does not belong to the spectrum of $\Dd_t^2$ for $t \in [0, t_0)$. Therefore the family of resolvents $s \mapsto  \Rr_{\Dd_t^2}^2 \lp - \lp s-\tfrac{1}{2} \rp^2  \rp$ is holomorphic on a neighborhood of $1/2$, for any $t \in [0,t_0)$. Consider the Taylor asymptotic expansion of the logarithmic derivative of the Selberg zeta function in $1/2$:
\begin{equation}\label{dj} 
\Dl = \sum_{j \geq 0} d_j(t) \lp s-\tfrac{1}{2} \rp^j. 
\end{equation}
As in \cite[Equation (2.15)]{sarnak}, let us denote 
\begin{equation}\label{cj}
\frac{d}{ds}\left[\frac{1}{2s-1} \Dl \right]  = \sum_{j \geq -2} c_j(t) \lp s-\tfrac{1}{2} \rp^j,  
\end{equation}
then it is easy to see that for any $j \neq 1$, 
\begin{equation} \label{dc}
d_{j+2}(t)=\frac{2}{j+1} c_j(t). 
\end{equation}

\begin{corollary}
There exists $t_0>0$ so that for all $j \in \{ 0 \} \cup \mathbb N \setminus \{ 1\}$, the coefficients $d_j(t)$ of the logarithmic derivative of the Selberg zeta function defined in \eqref{dj} behave in the pinching process as
\[ d_j(t) \in \mathcal C^{\infty}[0,t_0) + t^4 \log t  \mathcal C^{\infty}[0,t_0) . \]
\end{corollary}
\begin{proof}
By \eqref{comportamentdeppoli}, it follows that the coefficients of $\frac{d}{ds}\left[\frac{1}{2s-1} \Dl \right]$ defined in \eqref{cj} have the following behavior as $t \searrow 0$:
\[ c_j(t) \in  \mathcal C^{\infty}[0, t_0) + t^4 \log t C^{\infty}[0, t_0) ,  \]
where $t_0>0$ is chosen as in the proof of the Theorem \ref{mainth}. Therefore, by \eqref{dc}, the coefficients $d_j(t)$ for $j \neq 1$ have the same behavior as $t \searrow 0$ in the pinching process.
\end{proof}

\section{Towards the behavior of the Selberg zeta function in the pinching process}

As in the proof of \cite[Theorem $24$]{rares}, for $\Re(s)>1$ we isolate in the absolutely convergent series $\Dl$ the primitive conjugacy classes of the oriented pinched geodesic $\gamma$ of length $t$ and of its inverse. Recall our standing assumption $\varepsilon(\gamma)=-1$:
\begin{equation}\label{derlog1}
\begin{aligned}
\Dl ={}&  \sum_{[\eta]} \sum_{n=1}^{\infty} \frac{l_t(\eta) \varepsilon^n(\eta) e^{-n l_t(\eta) \lp s- \tfrac{1}{2} \rp }}{ 2 \sinh \frac{n l_t (\eta)}{2}} \\
={}& 2 \cdot \sum_{n=1}^{\infty} \frac{(-1)^n t e^{-nt \lp s-\tfrac{1}{2} \rp }}{2 \sinh \tfrac{nt}{2}}       + \sum_{[\eta] \neq [\gamma]} \sum_{n=1}^{\infty} \frac{l_t(\eta) \varepsilon^n(\eta) e^{-n l_t(\eta) \lp s- \tfrac{1}{2} \rp }}{ 2 \sinh \frac{n l_t (\eta)}{2}}.  
\end{aligned}
\end{equation}

Assume that we could prove that for some particular value $s_0$ with $\Re s_0>1$, there exists $t_0>0$ so that the value art $s_0$ of the log-derivative of $Z_t$, $\lp \frac{Z_t'}{Z_t} \rp (s_0)$, is smooth for $t \in [0, t_0)$. Then combining again \eqref{selcomp1} and Proposition \ref{tracecp} as in the previous section, we would deduce that the Selberg zeta function $Z_t$ is of class $\mathcal C^3$ in $(t,s)\in[0,\infty)\times \cz$, and holomorphic in $s$.

A natural choice for such a special value could be $s=2$. For this particular $s$, we
can prove with some effort that the function 
\[ S(t) := \sum_{n=1}^{\infty} \frac{(-1)^n t e^{-\tfrac{3nt}{2}}}{2 \sinh \tfrac{nt}{2}}  \]
is smooth for $t\in [0,\infty)$, proving smoothness of the first term in \eqref{derlog1}. 
The contribution of those simple geodesics disjoint from $\gamma$ in the second term of \eqref{derlog1} can also be shown to be smooth in $t\geq 0$, by the same proof as in Section \ref{Section3}.
We are unable however to treat the terms corresponding to the geodesics crossing $\gamma$. The length of such geodesics grows to infinity in the pinching limit, but it is not so easy to uniformly bound their derivatives in $t$. We are thus unable to find the desired special value of $s$ for which the logarthmic derivative of the Zeta function is smooth at $t=0$.

Another special value for $s$ where we could attempt to prove smoothness would be the central value $s=1/2$. Using \cite[Theorem 30]{rares}, we get:
\begin{align*}
2 \lp \frac{Z'_0}{Z_0}\rp \lp \frac{1}{2} \rp ={}&  \Area X \lim_{s \to 1/2} \lp s-\tfrac{1}{2} \rp \tan (\pi s) +4k \log 2= \frac{- \Area X}{\pi} + 4k \log 2,
\end{align*}
thus 
\begin{align*}
\lp \frac{Z'_0}{Z_0}\rp \left(\frac{1}{2}\right) = 2(g-1 + k \log 2).
\end{align*}
Thus, the Selberg trace formula implies only that the logarithmic derivative of the zeta function at the central value $1/2$ is constant on the spin moduli space. This identity does not tell us anything  about the central value of the zeta function itself. We must therefore leave open the question of proving $C^k$ asymptotics of the zeta function for any $k>0$.


\begin{thebibliography}{99}

\bibitem{ars}
P.~Albin, F.~Rochon, D.~Sher, \emph{Resolvent, heat kernel and torsion under degeneration to fibered cusps}, Mem.\ Amer.\ Math.\ Soc.\ \textbf{269} (2021), no.\ 1314.

\bibitem{cipriana}
C.~Anghel, \emph{On the spectrum of the Dirac operator on degenerating Riemannian surfaces}, arXiv:2409.05616.

\bibitem{Bar}
C.~B\"ar, \emph{The Dirac operator on hyperbolic manifolds of finite volume}, J.\ Differential Geom.\ \textbf{54} (2000), no.\ 3, 439--488.

\bibitem{BolteStiepanSelbergForDirac}
J.~Bolte, H.~M.~Stiepan, \emph{The Selberg trace formula for Dirac operators}, J.\ Math.\ Phys.\ \textbf{47} (2006), no.\ 11.

\bibitem{bgaud}
J.~Bourguignon, P.~Gauduchon, \emph{Spineurs, opérateurs de Dirac et variations de métriques}, Comm.\ Math.\ Phys.\ {\bf 144} (1992), 581--599.

\bibitem{moira}
M.~Chas, \emph{Relations between word length, hyperbolic length and self-intersection number of curves on surfaces}, Recent advances in mathematics, 45–75, Ramanujan Math.\ Soc.\ Lect.\ Notes Ser., \textbf{21}, Mysore, 2015.

\bibitem{colbois-courtois}
B.~Colbois, G.~Courtois, \emph{Les valeurs propres inf\'erieures \`a 1/4 des surfaces de Riemann de petit rayon d’injectivit\'e}, Comment.\ Math.\ Helvet.\ {\bf 64}, 349--362 (1989).

\bibitem{eichler}
M.~Eichler, \emph{Grenzkreisgruppen und kettenbruchartige Algorithmen}, Acta Arithmetica \textbf{11}
(1965), no.\ 2, 169-–180.


\bibitem{fedosovapohl}
K.~Fedosova, A.~Pohl, \emph{Meromorphic continuation of Selberg zeta functions with twists having non-expanding cusp monodromy}, Selecta Math.\ \textbf{26} (2020), no.\ 1, Paper No.\ 9, 55 pp.


\bibitem{GMP}
C.~Guillarmou, S.~Moroianu, J.~Park, \emph{Eta invariant and Selberg zeta function of odd type over convex co-compact hyperbolic manifolds}, Adv.\ Math.\ {\bf 225} (2010), no.\ 5, 2464–-2516.

\bibitem{dhocker}
E.~D'Hoker, D.~H.~Phong, \emph{On determinants of Laplacians on Riemann surfaces}, Comm.\ Math.\ Phys.\ \textbf{104} (1986), no. 4, 537--545.

\bibitem{grieser}
D.~Grieser, \emph{Basics of the b-calculus}, Approaches to Singular Analysis. Operator Theory: Advances and Applications \textbf{125}, 30--84, Birkhauser, 2001.

\bibitem{Ji}
L.~Ji, \emph{Spectral degeneration of hyperbolic Riemann surfaces}, J.\ Differential Geom.\ \textbf{38} (1993), no. 2, 263--313.  

\bibitem{mcdonald}
P.~McDonald, \emph{The Laplacian for spaces with cone-like singularities}, PhD Thesis, MIT (1990).

\bibitem{melrose}
R.B.~Melrose, \emph{The Atiyah-Patodi-Singer Index Theorem}, Research Notes in Mathematics {\bf 4}, A.K.Peters, (1993). 

\bibitem{moroweyl}
S.~Moroianu, \emph{Weyl laws on open manifolds}, Math.\ Ann.\  {\bf 340} (2008), no.\ 1, 1--21. 

\bibitem{park}
J.~Park, \emph{Eta invariants and regularized determinants for odd dimensional hyperbolic manifolds with cusps},
Amer.\ J.\ Math.\ \textbf{127} (3), (2005), 493--534.

\bibitem{pfaffle}
F.~Pfäffle, \emph{Eigenvalues of Dirac operators for hyperbolic degenerations}, Manuscr.\ Math.\  {\bf 116} (2005), no. $1$, 1--29. 

\bibitem{sarnak}
P.~Sarnak, \emph{Determinants of Laplacians}, Comm.\ Math.\ Phys.\ \textbf{110} (1987), no. 1, 113--120.

\bibitem{Schulze}
M.~Schulze, \emph{On the resolvent of the Laplacian on functions for degenerating surfaces of finite geometry},
J.\ Funct.\ Anal.\ \textbf{236} (2006), no.\ 1, 120--160.

\bibitem{selberg}
A.~Selberg, \emph{Harmonic analysis and discontinuous groups in weakly symmetric Riemannian spaces with applications to Dirichlet series}, J.\ Indian Math.\ Soc.\ \textbf{20} (1956), 47--87.

\bibitem{series}
C.~Series, \emph{An extension of Wolpert’s derivative formula},
Pacific J.\ Math.\ \textbf{197} (2001), 223–239.

\bibitem{rares}
R.~Stan, \emph{The Selberg trace formula for spin Dirac operators on degenerating hyperbolic surfaces}, Doc.\ Math.\ \textbf{30} (2025), 1--39.

\bibitem{wolpert}
S.~Wolpert, \emph{Behavior of geodesic-length functions on Teichmüller space},
J.\ Differential Geom.\ \textbf{79} (2008), no. 2, 277–334.

\end{thebibliography}
\end{document}